\documentclass[reqno, 12pt]{amsart}
\pdfoutput=1
\makeatletter
\let\origsection=\section \def\section{\@ifstar{\origsection*}{\mysection}} 
\def\mysection{\@startsection{section}{1}\z@{.7\linespacing\@plus\linespacing}{.5\linespacing}{\normalfont\scshape\centering\S}}
\makeatother        

\usepackage{amsmath,amssymb,amsthm}
\usepackage{mathrsfs}
\usepackage{mathabx}\changenotsign
\usepackage{dsfont}
 
\usepackage{xcolor}
\usepackage[backref]{hyperref}
\hypersetup{
    colorlinks,
    linkcolor={red!60!black},
    citecolor={green!60!black},
    urlcolor={blue!60!black}
}

\usepackage[open,openlevel=2,atend]{bookmark}

\usepackage[abbrev,msc-links,backrefs]{amsrefs} 
\usepackage{doi}

\renewcommand{\PrintDOI}[1]{\doi{#1}}

\usepackage[T1]{fontenc}
\usepackage{lmodern}
\usepackage[babel]{microtype}
\usepackage[english]{babel}

\usepackage{geometry}
\numberwithin{equation}{section}
\numberwithin{figure}{section}
\usepackage{tikz}
\usepackage{xcolor}

\usepackage{enumitem}
\def\rmlabel{\upshape({\itshape \roman*\,})}

\def\alabel{\upshape({\itshape \alph*\,})}

\def\nlabel{\upshape({\itshape \arabic*\,})}

\makeatletter
\def\greek#1{\expandafter\@greek\csname c@#1\endcsname}
\def\Greek#1{\expandafter\@Greek\csname c@#1\endcsname}
\def\@greek#1{\ifcase#1
	\or $\alpha$%
	\or $\beta$%
	\or $\gamma$%
	\or $\delta$%
	\or $\epsilon$%
	\or $\zeta$%
	\or $\eta$%
	\or $\theta$%
	\or $\iota$%
	\or $\kappa$%
	\or $\lambda$%
	\or $\mu$%
	\or $\nu$%
	\or $\xi$%
	\or $o$%
	\or $\pi$%
	\or $\rho$%
	\or $\sigma$%
	\or $\tau$%
	\or $\upsilon$%
	\or $\phi$%
	\or $\chi$%
	\or $\psi$%
	\or $\omega$%
\fi}
\def\@Greek#1{\ifcase#1
	\or $\mathrm{A}$%
	\or $\mathrm{B}$%
	\or $\Gamma$%
	\or $\Delta$%
	\or $\mathrm{E}$%
	\or $\mathrm{Z}$%
	\or $\mathrm{H}$%
	\or $\Theta$%
	\or $\mathrm{I}$%
	\or $\mathrm{K}$%
	\or $\Lambda$%
	\or $\mathrm{M}$%
	\or $\mathrm{N}$%
	\or $\Xi$%
	\or $\mathrm{O}$%
	\or $\Pi$%
	\or $\mathrm{P}$%
	\or $\Sigma$%
	\or $\mathrm{T}$%
	\or $\mathrm{Y}$%
	\or $\Phi$%
	\or $\mathrm{X}$%
	\or $\Psi$%
	\or $\Omega$%
\fi}

\AddEnumerateCounter{\greek}{\@greek}{24}
\AddEnumerateCounter{\Greek}{\@Greek}{12}

\makeatother

\def\glabel{\upshape({\itshape \greek*})}

\let\polishlcross=\l
\def\l{\ifmmode\ell\else\polishlcross\fi}

\def\qand{\quad\text{and}\quad}

\let\sm=\setminus

\makeatletter
\def\moverlay{\mathpalette\mov@rlay}
\def\mov@rlay#1#2{\leavevmode\vtop{   \baselineskip\z@skip \lineskiplimit-\maxdimen
   \ialign{\hfil$\m@th#1##$\hfil\cr#2\crcr}}}
\newcommand{\charfusion}[3][\mathord]{
    #1{\ifx#1\mathop\vphantom{#2}\fi
        \mathpalette\mov@rlay{#2\cr#3}
      }
    \ifx#1\mathop\expandafter\displaylimits\fi}
\makeatother

\newcommand{\dcup}{\charfusion[\mathbin]{\cup}{\cdot}}

\DeclareFontFamily{U}  {MnSymbolC}{}
\DeclareSymbolFont{MnSyC}         {U}  {MnSymbolC}{m}{n}
\DeclareFontShape{U}{MnSymbolC}{m}{n}{
    <-6>  MnSymbolC5
   <6-7>  MnSymbolC6
   <7-8>  MnSymbolC7
   <8-9>  MnSymbolC8
   <9-10> MnSymbolC9
  <10-12> MnSymbolC10
  <12->   MnSymbolC12}{}
\DeclareMathSymbol{\powerset}{\mathord}{MnSyC}{180}

\let\epsilon=\varepsilon
\let\eps=\epsilon
\let\rho=\varrho
\let\theta=\vartheta
\let\kappa=\varkappa

\def\NN{{\mathds N}}

\def\RR{{\mathds R}}

\def\QQ{{\mathds Q}}

\def\RT{\mathrm{RT}}
\def\ex{\mathrm{ex}}

\newcommand{\ccF}{\mathscr{F}}

\theoremstyle{plain}
\newtheorem{thm}{Theorem}[section]
\newtheorem{prop}[thm]{Proposition}
\newtheorem{clm}[thm]{Claim}
\newtheorem{fact}{Fact}[thm]
\newtheorem{fakt}[thm]{Fact}
\newtheorem{cor}[thm]{Corollary}
\newtheorem{lemma}[thm]{Lemma}

\theoremstyle{definition}

\newtheorem{dfn}[thm]{Definition}

\usepackage{accents}

\let\phi=\varphi

\begin{document}

\title[The Ramsey-Tur\'an problem for cliques]
{The Ramsey-Tur\'an problem for cliques}

\dedicatory{Dedicated to Vera T. S\'os}
\author[Clara~M.~L\"uders]{Clara Marie L\"uders} 
\author[Christian Reiher]{Christian Reiher}
\address{Fachbereich Mathematik, Universit\"at Hamburg, Hamburg, Germany}
\email{Christian.Reiher@uni-hamburg.de}
\email{Clara.Marie.Lueders@gmail.com}

\subjclass[2010]{05C35}
\keywords{Ramsey-Tur\'an problem, cliques}

\begin{abstract}
An important question in extremal graph theory raised by Vera T. S\'os asks to determine 
for a given integer $t\ge 3$ and a given positive real number $\delta$ the asymptotically
supremal edge density $f_t(\delta)$ that an $n$-vertex graph can have provided it contains 
neither a complete graph $K_t$ nor an independent set of size $\delta n$. 

Building upon recent work of Fox, Loh, and Zhao [{\it The critical window for the classical 
Ramsey-Tur\'an problem}, Combinatorica {\bf 35} (2015), 435--476], we prove that if 
$\delta$ is sufficiently small (in a sense depending on $t$), then
\[
	f_t(\delta)=
		\begin{cases}
			\frac{3t-10}{3t-4}+\delta-\delta^2 & \text{ if $t$ is even,} \cr
			\frac{t-3}{t-1}+\delta & \text{ if $t$ is odd.}
		\end{cases}
\]
\end{abstract}

\maketitle

\section{Introduction}
P.~Tur\'an~\cite{Turan} established a new subarea of extremal combinatorics
nowadays bearing his name. In the context of graphs, the fundamental question he proposed is to
determine, for a given positive number~$n$ and a given graph~$F$, the maximum number~$\ex(n, F)$
of edges that a graph of order $n$ can have provided that it does not contain~$F$ 
as a subgraph. Tur\'an himself gave the complete answer if $F$ is a clique, and an
asymptotically satisfactory solution for all graphs $F$ has been obtained by the work of 
Erd\H{o}s, Stone, and Simonovits (see~\cites{ErSt46, ErSi66}). 
Curiously, the corresponding problem for hypergraphs
is wide open, even in the $3$-uniform case.

Another branch of combinatorics related to our discussion, called Ramsey theory, was initiated 
by F.~P.~Ramsey~\cite{Ramsey30} and since then it has been developed into a coherent and 
successful body of results. 
A somewhat special yet typical case of Ramsey's original theorem asserts that if $n$ is large 
enough depending on $k$, then no matter how one colours the edges of a complete graph of 
order $n$ using two colours, there will always be a monochromatic complete subgraph of 
order $k$.
 
Vera T.~S\'os discovered a beautiful way of combining Ramsey theory with Tur\'an theory by 
asking and investigating the following question: Given a positive integer $n$, 
a positive real number $m$, and a graph~$F$, what is the maximum number $\RT(n, m, F)$ of 
edges that a graph~$G$ of order $n$ can have if it does not contain $F$ as a subgraph 
and $\alpha(G)<m$, 
i.e., if any~$X\subseteq V(G)$ with $|X|\ge m$ spans at least one edge? 

For example, if $m=n+1$ and $F$ has at least one edge, then the condition on independent 
sets becomes vacuous and one recovers Tur\'an's original problem, i.e., 
one has $\RT(n, n, F)=\ex(n, F)$. 
On the other hand, if~$m$ is very small, then by Ramsey's theorem each graph of order~$n$
contains either a clique of order~$v(F)$ (and hence, in particular, a subgraph isomorphic 
to~$F$) or an independent set of order~$\lceil m\rceil$, meaning that the definition 
of~$\RT(n, m, F)$ degenerates to the ``maximum of the empty set.'' 
Using a quantitative version of Ramsey's theorem, this can be seen to happen, 
e.g., if $m<n^{1/v(F)}$ and $n$ is large. 
So for fixed $n$ and $F$ the problem of determining $\RT(n, m, F)$ is mostly dominated
by Ramsey theoretic phenomena for very small $m$ and by Tur\'an theory for very large~$m$.
If $m$ is of medium size, however, the problem intriguingly combines the flavours of 
both areas. For further information on Ramsey-Tur\'an theory the reader is referred to the 
comprehensive survey~\cite{SS01} by Simonovits and S\'os. 

\medskip

In this article we restrict our attention to the perhaps most classical case 
that $m=\delta n$ for some small~$\delta>0$ and $F=K_t$ is a clique. To eliminate minor fluctuations arising from small
values of~$n$ one usually focuses on the {\it Ramsey-Tur\'an density function}
$f_t\colon (0, 1)\longrightarrow\RR$ defined by 
\begin{equation*}
	f_t(\delta)=\lim_{n\to\infty}\frac{\RT(n, \delta n, K_t)}{n^2/2}\,. 
\end{equation*}

It is well known and easy to confirm that this limit does indeed exist. 
Since $f_t$ is evidently a nondecreasing function of $\delta$, a further simplification 
may be achieved by passing to the {\it Ramsey-Tur\'an density} $\rho(K_t)$ 
defined by 
\[
	\rho(K_t)=\lim_{\delta\to 0} f_t(\delta)\,.
\]

Perhaps surprisingly at first, the difficulty of determining the quantities just introduced 
depends significantly on the parity of $t$. The first case where something happens
is $t=3$. One has $\RT(n, \delta n, K_3)\le\delta n^2/2$ because if a graph $G$ of order $n$
has a vertex $x$ whose degree is at least $\delta n$, then either the neighbourhood of $x$ is 
independent, which gives $\alpha(G)\ge \delta n$, or this neighbourhood spans an edge $yz$,
in which case $xyz$ is a triangle. This simple observation implies $f_3(\delta)\le \delta$
for all $\delta>0$. Explicit examples described by Brandt~\cite{B10} show that 
for~$\delta<\tfrac13$
this bound is optimal (see Proposition~\ref{prop:Brandt} and also 
Corollary~\ref{cor:odd-lower} below), i.e., that we have
$f_3(\delta)=\delta$ for all $\delta\in \bigl(0, \tfrac13\bigr)$; in particular, $\rho(K_3)=0$.
Concerning larger odd cliques, Erd\H{o}s and S\'os~\cite{ES69} proved 
$\rho(K_{2r+1})=\tfrac{r-1}r$ for 
all positive integers $r$, and a quantitative version of their argument yields 
\[
	\frac{r-1}r \le f_{2r+1}(\delta) \le \frac{r-1}r + 2\delta
\]
for all positive $\delta$. 

The first result addressing an even clique was obtained by Szemer\'edi~\cite{Sz72}, 
who proved that $\rho(K_4)\le \tfrac14$. At that moment it still seemed conceivable that the
truth might be ${\rho(K_4)=0}$. But a few years later Bollob\'as and Erd\H{o}s~\cite{BE76}
ruled out this possibility by exhibiting a remarkable geometric construction 
demonstrating the optimality of Szemer\'edi's bound; that is they completed the proof of 
$\rho(K_4)= \tfrac14$. 
Still later the Ramsey-Tur\'an densities of all even cliques were determined by 
Erd\H{o}s, Hajnal, S\'os, and Szemer\'edi~\cite{EHSS},
the answer being 
\begin{equation}\label{eq:EHSS}
	\rho(K_{2r})=\tfrac{3r-5}{3r-2} \quad \text{ for all } r\ge 2\,.
\end{equation}  

The understanding as to how fast $f_4(\delta)$ converges to $\tfrac 14$
developed as follows. Szemer\'edi's original argument yields
\[
	f_4(\delta)\le \tfrac14+O\left(\bigl(\log\log\tfrac1{\delta}\bigr)^{-1/2+o(1)}\right)\,.
\]
Conlon and Schacht observed independently in unpublished work that the Frieze-Kannan 
regularity lemma from~\cite{FK99} can be used to improve this to
\[
	f_4(\delta)\le \tfrac14+O\left(\bigl(\log\tfrac1{\delta}\bigr)^{-1/2}\right)\,.
\]
Significant further progress is due to Fox, Loh, and Zhao~\cite{FLZ15}, who obtained
\begin{equation}\label{eq:FLZ}
	\tfrac14+\delta-\delta^2\le f_4(\delta)\le\tfrac14 + 3\delta
\end{equation}
for sufficiently small $\delta$ and asked 
\begin{enumerate}[label=\nlabel]
\item how this gap can be narrowed down further 
\item and whether comparable results could be proved for larger even cliques
and, in particular,  
whether $f_{2r}(\delta)=\rho(K_{2r})+\Theta(\delta)$ holds for all $r\ge 2$. 
\end{enumerate} 

Our main result addresses both questions. Much to our own surprise, it turned out
that at least for $\delta\ll r^{-1}$ there is a precise formula for the values of the 
Ramsey-Tur\'an density function.

\begin{thm}\label{thm:even}
	If $r\ge 2$ and $\delta\ll r^{-1}$, then 
	$f_{2r}(\delta)=\tfrac{3r-5}{3r-2}+\delta-\delta^2$. 
\end{thm} 

The hard part of this result is the upper bound and we would like to restate it 
here in an elementary form, i.e., without talking about the function $f_{2r}$.

\enlargethispage{2em}
\begin{thm}\label{thm:main}
	For every integer $r\ge 2$ there exists a real number $\delta_*>0$ such that 
	if~$\delta\le \delta_*$, then every graph $G$ on $n$ vertices with 
		\[
		\alpha(G)<\delta n \qand e(G)> 
		\bigl(\tfrac{3r-5}{3r-2}+\delta-\delta^2\bigr)\tfrac{n^2}2
	\]
		contains a $K_{2r}$.  
\end{thm}

Incidentally, such an exact formula does also hold for odd cliques. 

\begin{thm}\label{thm:odd}
	If $r\ge 1$ and $\delta\ll r^{-1}$, then $f_{2r+1}(\delta)=\tfrac{r-1}r+\delta$.
\end{thm}

\subsection*{Organisation} The lower bound constructions establishing that $f_t(\delta)$
has at least the value claimed in Theorem~\ref{thm:even} and Theorem~\ref{thm:odd} are given
in Section~\ref{sec:constructions}. 
The upper bound for the Ramsey Tur\'an density function of odd cliques is proved in 
Section~\ref{sec:odd}. The proof of Theorem~\ref{thm:main} constitutes the 
main part of this article and occupies the Sections~\ref{sec:even-intro}--\ref{sec:refine}. 

\section{The lower bounds} 
\label{sec:constructions}

The goal of this section is to verify the lower bounds on $f_t(\delta)$ 
from Theorem~\ref{thm:even} and Theorem~\ref{thm:odd} by means of explicit 
constructions. To this end, we just need to combine some results from~\cite{B10}
and~\cite{FLZ15}. 

We begin by recapitulating~\cite{B10}*{Theorem~2.1}. This statement deals with 
the set $\Omega$ of all pairs $(d, n)$ of natural numbers for which there exists 
a triangle-free, $d$-regular graph on~$n$ vertices with independence number~$d$. 
Of course, if $(d, n)\in \Omega$, then $\RT(n, d+1, K_3)=\frac 12 dn$ is 
as large as possible. 

A standard blow-up argument shows that if
$(d, n)\in\Omega$, then all multiples of this pair belong to $\Omega$ as well, 
that is we have $(ad, an)\in \Omega$ for all $a\in\NN$. This suggest that rather 
than studying $\Omega$ itself one may want to focus on the set of quotients
\[
	S=\Bigl\{\tfrac dn\colon (d, n)\in\Omega\Bigr\}\,.
\]
Brandt~\cite{B10} discovered constructions which show the following. 

\begin{prop}\label{prop:Brandt}
The set $S\cap \bigl(0, \tfrac13\bigr)$ is dense in $\bigl(0, \tfrac13\bigr)$.
Moreover, $\bigl(0, \tfrac{7}{30}\bigr)\cap\QQ$ and $\bigl(\tfrac14, \tfrac13\bigr)\cap\QQ$
are subsets of $S$.
\end{prop}

The ``moreover''-part is not going to be used in the sequel and it has been included here
for the readers information only.

\begin{cor}\label{cor:odd-lower}
For fixed $r\ge 1$ and $\delta<\tfrac{1}{3r}$ we have 
\[
	\RT(n, \delta n, K_{2r+1})\ge \bigl(\tfrac{r-1}{r}+\delta-o(1)\bigr)\tfrac{n^2}{2}\,.
\]
\end{cor}

\begin{proof}
	Let $\eta>0$ be given. We need to show that 
	$\RT(n, \delta n, K_{2r+1})\ge \bigl(\tfrac{r-1}{r}+\delta-\eta\bigr)\tfrac{n^2}{2}$
	holds for all sufficiently large integers $n$. By Proposition~\ref{prop:Brandt}
	there exists a pair 
	$(d_*, n_*)\in\Omega$ such that 
	$\frac{d_*}{n_*}\in \bigl(r(\delta-\eta), r\delta\bigr)$.
	Now it suffices to show that 
		\begin{equation}\label{eq:arn}
		\RT(arn_*, ad_*+1, K_{2r+1})\ge \left(\frac{r-1}{r}
			+\frac{d_*}{rn_*}\right)\frac{(arn_*)^2}2
	\end{equation}
		holds for every~$a\in\NN$. This is because for sufficiently large $n$ we can add
	at most $rn_*$ isolated vertices to a graph establishing~\eqref{eq:arn},
	thus obtaining the desired lower bound on $\RT(n, \delta n, K_{2r+1})$.
	
	To prove~\eqref{eq:arn} we use $(ad_*, an_*)\in\Omega$ and take a triangle-free,
	$(ad_*)$-regular graph $H$ on~$an_*$ vertices with $\alpha(H)=ad_*$.
	Now let $V=V_1\dcup \ldots\dcup V_r$ be a disjoint union of $r$ vertex classes 
	each of which has size $an_*$, and construct a graph $G$ on $V$ 
	\begin{enumerate}
	\item[$\bullet$] inducing on each vertex class $V_i$ a graph isomorphic to $H$,
	\item[$\bullet$] in which any two vertices from different classes are adjacent.  
	\end{enumerate}  
	
	From $K_3\not\subseteq H$ and the box principle it follows that $K_{2r+1}\not\subseteq G$.
	Every subset of $V$ which is independent in $G$ needs to be contained in a single
	vertex class, whence 
		\[
		\alpha(G)=\alpha(H)<ad_*+1\,.
	\]
		Finally, we have  
		\[
		e(G)=\binom{r}{2}(an_*)^2+re(H)=\left(\frac{r-1}{r}
			+\frac{d_*}{rn_*}\right)\frac{(arn_*)^2}2\,. 
	\]
		Therefore, $G$ has all the properties necessary for witnessing~\eqref{eq:arn}. 
\end{proof}

Let us proceed with essentially extremal examples for even cliques. 
As mentioned in the introduction, Bollob\'as and Erd\H{o}s~\cite{BE76} found a geometric  
construction showing that $\RT(n, o(n), K_4)\ge \bigl(\frac 14+o(1)\bigr)\frac{n^2}{2}$. The 
vertex set of their graph splits into two subsets of size~$\frac n2$ inducing 
triangle-free graphs with $o(n^2)$ edges. Between those sets, called 
$A$ and~$B$ from now on, there is a very special quasirandom bipartite graph 
of density $\frac12-o(1)$. 

To aid the readers orientation we remark that the graphs induced by $A$ and $B$ 
are not only triangle-free. As a matter of fact, they are ``locally bipartite'' 
in the sense of having rather large odd-girth. In particular, they do not contain cycles of 
length $5$ or $7$. Such properties will also play an important r\^{o}le in our 
proof of the upper bound (see Fact~\ref{fact:C7} below).

It is not entirely straightforward to make the asymptotic expressions in the result
of Bollob\'as and Erd\H{o}s explicit. The best quantitative analysis we are 
aware of has been conducted by Fox, Loh, and Zhao~\cite{FLZ15}*{Corollary~8.9},
who obtained the following.  
 
\begin{thm}\label{thm:FLZ1}
	If $n$ is sufficiently large and $\xi=4(\log\log n)^{3/2}/(\log n)^{1/2}$,
	then 
		\[
		\RT(n, \xi n, K_4)\ge \bigl(\tfrac18-\xi\bigr)n^2\,.
	\]
	\end{thm}

Let us proceed with a discussion of~\cite{FLZ15}*{Theorem~1.7} and the remark thereafter.  
Suppose that $\delta\in\bigl(0, \tfrac 12\bigr)$ is fixed and that $n$ is a sufficiently 
large and (just for transparency) even natural number. 
Let $G$ be a graph on $n$ vertices as obtained by Theorem~\ref{thm:FLZ1}.
Recall that there is a partition $V(G)=A\dcup B$ with $|A|=|B|=\frac n2$ of its vertex set
into two subsets not inducing triangles. Let $X\subseteq A$ and $Y\subseteq B$ be
two random sets of size $|X|=|Y|=(\delta-\xi)n$, and let $G_*$ be the graph obtained 
from $G$ by removing all edges incident with $X\cup Y$ and then adding all edges 
from $X$ to $B$ as well as all edges from $Y$ to~$A$. Surely, $G_*$ is $K_4$-free 
and all its independent sets have size less than $\delta n$. Moreover, a short 
calculation displayed in the proof of~\cite{FLZ15}*{Lemma~9.1} shows that the 
expected number of edges of $G_*$ is at least 
$\bigl(\tfrac14+\delta-\delta^2-o(1)\bigr)\tfrac{n^2}2$. Therefore, we have indeed 
$f_4(\delta)\ge \tfrac14+\delta-\delta^2$. 

This construction combines with~\cite{EHSS}*{Theorem~5.4} in the following way.

\begin{prop}\label{prop:24}
	If $r\ge 2$ and $\delta\in\bigl(0, \tfrac{1}{3r-2}\bigr)$ are fixed, 
	then 
		\[
		\RT(n, \delta n, K_{2r})
			\ge \bigl(\tfrac{3r-5}{3r-2}+\delta-\delta^2-o(1)\bigr)\tfrac{n^2}2\,.
	\]
	\end{prop}   

\begin{proof} 
	Let $n$ be sufficiently large and, without loss of generality, divisible by $3r-2$.
	Take a set $V$ of $n$ vertices as well as a partition 
		\begin{equation}\label{eq:Vi}
		V=V_1\dcup V_2\dcup\ldots\dcup V_r
	\end{equation}
		with $|V_i|=\frac{2}{3r-2}n$ for $i=1, 2$ and $|V_i|=\frac{3}{3r-2}n$ for $i=3, \ldots, r$.
	Construct a graph $G$ on~$V$ whose edges are as follows.
	\begin{enumerate}
		\item[$\bullet$] The subgraph of $G$ induced by $V_1\dcup V_2$ 
			is the graph described above exemplifying the lower bound 
						\[
				\RT\bigl(\tfrac 4{3r-2}n, \delta n, K_4\bigr)\ge 
				\tfrac{2}{(3r-2)^2}n^2+\tfrac2{3r-2}\delta n^2-\tfrac12\delta^2n^2-o(n^2)\,,
			\]
						the sets $V_1$ and $V_2$ here playing the r\^{o}les of $A$ and $B$ there.   
		\item[$\bullet$] For $i\in [3, r]$ the graph that $G$ induces on $V_i$ is obtained 
			by Corollary~\ref{cor:odd-lower} and demonstrates
						\[
				\RT\bigl(\tfrac 3{3r-2}n, \delta n, K_3\bigr)\ge \tfrac{3}{2(3r-2)}\delta n^2
					-o(n^2)\,.
			\]
					\item[$\bullet$] If $1\le i<j\le r$ and $(i, j)\ne (1, 2)$, then all pairs 
			$uv$ with $u\in V_i$ and $v\in V_j$ are edges of $G$.
	\end{enumerate} 
	
	Evidently, every clique in $G$ can have at most three vertices in $V_1\dcup V_2$
	and at most two vertices in each $V_i$ with $i\in [3, r]$, which proves that $G$ 
	is $K_{2r}$-free. Moreover, each independent subset of $V$ is either contained 
	in $V_1\dcup V_2$ or in one of the sets $V_i$ with $i\in [3, r]$. Consequently, 
	we have $\alpha(G)<\delta n$. Finally, a quick computation shows 
		\begin{align*}
		2e(G) &= \bigl[ \bigl( \tfrac{4}{(3r-2)^2} +\tfrac{4}{3r-2}\delta-\delta^2 \bigr)  
			+\tfrac{3(r-2)}{3r-2}\delta+\tfrac{9(r-2)(r-3)+24(r-2)}{(3r-2)^2}-o(1)\bigr]n^2 \\
			&= \bigl( \tfrac{3r-5}{3r-2}+\delta-\delta^2-o(1)\bigr)n^2\,.
	\end{align*}
		So altogether $G$ has all required properties. 
\end{proof}

\section{Odd cliques} \label{sec:odd}

\subsection{Overview} 
This section deals with the proof of Theorem~\ref{thm:odd}. As the lower bound has already 
been established in Corollary~\ref{cor:odd-lower}, it will suffice to prove the following result.

\begin{thm} \label{thm:rt-odd}
	Suppose that $r$ is a positive integer and $0<\delta<\frac{1}{289r^5}$. If $G$
	is a $K_{2r+1}$-free graph with $n$ vertices 
	and $\alpha(G)<\delta n$, then $e(G)\le \bigl(\frac{r-1}{r}+\delta\bigr)\frac{n^2}{2}$. 
\end{thm}

Before coming to the details we would like to give an informal description of the main idea
occurring in the proof of Theorem~\ref{thm:rt-odd}. 
First of all, it suffices to prove this result, for a somewhat larger range of $\delta$, under
the minimum degree assumption $\delta(G)\ge \frac{r-1}{r}$, for then standard arguments 
allow us to infer the general statement (see Proposition~\ref{prop-odd-mindeg} below).
Next it can be proved in a rather precise sense that graphs fulfilling this minimum degree 
condition and the other assumptions of Theorem~\ref{thm:rt-odd} need to look almost like the 
graphs presented in the proof of Corollary~\ref{cor:odd-lower}. In particular, the edges of 
such graphs can be coloured {\it red} and {\it green} in such a way that 
\begin{enumerate}
	\item[\tikz{\shade[ball color = red] (0,0) circle (3pt)}]
		the red graph is $K_{r+1}$-free
	\item[\tikz{\shade[ball color = green] (0,0) circle (3pt)}]
		and the green graph has maximum degree $\delta n$.
\end{enumerate} 

In the extremal construction, the red graph was actually an $r$-partite Tur\'an graph, 
while the green graph was the disjoint union of $r$ triangle-free graphs each of which 
had $n/r$ vertices. Applying Tur\'an's theorem to the red part and the inequality 
$e(G)\le \Delta(G)v(G)/2$ to the green part one checks easily that every graph
admitting an edge colouring with the two properties above has at 
most $\bigl(\frac{r-1}{r}+\delta\bigr)\frac{n^2}{2}$ edges.   

We are thus left with the task of colouring the edges of every graph $G$ as in 
Theorem~\ref{thm:rt-odd} and having large minimum degree in the the desired way. 
Now in the extremal case the joint neighbourhood of a red edge has size 
$\bigl(\frac{r-2}{r}+2\delta)n$,
which is considerably less than the corresponding value of about $\frac{r-1}{r}n$ for 
green edges. For the general case this suggests to {\it define} an edge to be red 
if its joint neighbourhood is ``small'' and green otherwise, and in fact this is what we 
shall do later in the proof of Proposition~\ref{prop-odd-mindeg}.
 
\subsection{Preparations}
We begin with a result saying that among any $r+1$ large-degree vertices in a graph there 
is a always a pair whose joint neighbourhood is ``large.'' This will be used later for 
excluding red cliques of order $r+1$.  

\begin{lemma}
	Given a graph $G=(V, E)$ on $n$ vertices and a set $Q\subseteq V$ with $|Q|=r+1\ge 2$,
	there exist distinct $x, y\in Q$ with 
	$|N(x)\cap N(y)|\ge \frac{r-1}{r}\bigl(d(x)+d(y)\bigr)-\frac{r-1}{r+1}n$.	 
\end{lemma}

\begin{proof}
	Notice that for every integer $k$ with $0\le k\le r+1$ we have 
	\[
		k(r-1)-\tbinom{r}{2}=\tbinom{k}{2}-\tbinom{r-k}{2}\le \tbinom{k}{2}\,.
	\]

	Thus writing $Q^{(2)}$ for the collection of all two-element subsets of $Q$ and $W_k$
	for the set of all vertices in $V$ with exactly $k$ neighbours in $Q$ we have 
	\begin{align*}
		& \sum_{xy\in Q^{(2)}} \left(\tfrac{r-1}{r}\bigl(d(x)+d(y)\bigr)-\tfrac{r-1}{r+1}n\right)
			=(r-1)\sum_{x\in Q}d(x)-\tbinom{r}{2}n \\
		&=\sum_{k=0}^{r+1} \left(k(r-1)-\tbinom{r}{2}\right) |Q_k|
		\le \sum_{k=0}^{r+1} \tbinom{k}{2} |Q_k|
		=\sum_{xy\in Q^{(2)}} |N(x)\cap N(y)|\,,
	\end{align*}
	from which the desired result follows immediately.  
\end{proof}

In view of Tur\'an's theorem, this has the following consequence. 

\begin{cor}\label{cor:red}
	If $G$ is a graph on $n$ vertices, then for every positive integer $r$ there are
	at most $\frac{r-1}{2r}n^2$ edges $xy\in E(G)$ 
	with $|N(x)\cap N(y)|<\frac{r-1}{r}\bigl(d(x)+d(y)\bigr)-\frac{r-1}{r+1}n$. \hfill $\Box$
\end{cor}

The next lemma collects some facts about edge-maximal $K_{2r+1}$-free graphs with large 
minimum degree and small independence number. 

\begin{lemma}\label{lem:max}
	Let $r\ge 2$ and $0<\delta<\frac{1}{2r}$. Suppose that $G$ is an edge-maximal 
	$K_{2r+1}$-free graph
	on $n$ vertices with $\alpha(G)<\delta n$ and $\delta(G)\ge \frac{r-1}{r}n$.
	\begin{enumerate}[label=\rmlabel]
		\item\label{it:i} We have $\Delta(G) < \bigl(\frac{r-1}{r}+2r\delta\bigr)n$.
		\item\label{it:ii} Every $Q\subseteq V(G)$ with 
			$|Q|\ge \bigl(\frac{2r-3}{2r}+r\delta\bigr)n$ induces a $K_{2r-2}$.
		\item\label{it:iii} If an edge $xy$ of $G$ satisfies $N(x)\cup N(y)\ne V(G)$,
			then
			\[
				|N(x)\cap N(y)|\ge d(x)+d(y)-\bigl(\tfrac{r-1}{r}+8r\delta\bigr)n\,.
			\]
	\end{enumerate}
\end{lemma}

\begin{proof}
	Notice that $\delta n>\alpha(G)\ge 1$ and our upper bound on $\delta$ entail $n>2r$. 
	Thus the maximality of $G$ among $K_{2r+1}$-free graphs on $V(G)$ implies that every 
	vertex of $G$ is in a~$K_{2r}$. 
	
	For the proof of~\ref{it:i} we consider an arbitrary vertex $x\in V(G)$ and let 
	$T$ denote the vertex set of a $K_{2r}$ in $G$ containing $x$. For every $t\in T$
	the joint neighbourhood of $T\sm\{t\}$ is an independent set, since otherwise $G$ 
	would contain a $K_{2r+1}$. Consequently, each of these joint neighbourhoods contains 
	fewer than $\delta n$ vertices, whence 
	\[
		\sum_{t\in T}d(t)<(2r-2)n+2r\delta n\,.
	\]
	Taking the minimum degree condition on $G$ into account we deduce 
	$d(x)<(\frac{r-1}{r}+2r\delta\bigr)n$ and, as $x$ was arbitrary,~\ref{it:i} follows. 
	
	For the proof of~\ref{it:ii} we remark that the subgraph of $G$ induced by $Q$ 
	has minimum degree at least $|Q|-\frac nr$. Let $s\ge 2$ be maximal such that this graph 
	contains a $K_s$ and let~$Z$ denote the vertex set of some $K_s$ in $G$. By the same 
	argument as above we obtain
	\[
		s(|Q|-\tfrac nr)\le \sum_{z\in Z}|N(z)\cap Q|< (s-2)|Q|+s\delta n
	\]
	and thus 
	\[
		 \bigl(\tfrac{2r-3}{r}+2r\delta\bigr)n\le 2|Q|<\tfrac sr n+s\delta n\,,
	\]
	which is incompatible with $s\le 2r-3$. In other words, $Q$ contains indeed a $K_{2r-2}$.
	
	Preparing the proof of~\ref{it:iii} we show first that if $v$ and $w$ are distinct
	vertices of $G$ with $vw\not\in E(G)$, then 
	\begin{equation}\label{eq:vw}
		|N(v)\sm N(w)|\le \bigl(\tfrac{r-1}{r}+4r\delta\bigr)n-d(w)\,.
	\end{equation}
	To this end we use the edge-maximality of $G$, which gives us a $K_{2r-1}$ in $G$ 
	whose joint neighbourhood contains $v$ and $w$. Denote the vertex set on some such clique 
	by $A$ and let~$J$ be the set of all those vertices which have at most $2r-3$ neighbours 
	in $A$. Exploiting that the joint neighbourhood of $A$ can contain at most $\delta n$
	vertices we obtain
	\[
		\tfrac{(2r-1)(r-1)}{r}n\le \sum_{a\in A}d(a)\le (2r-3)n+|V(G)\sm J|+\delta n\,,
	\]
	i.e., $|J|\le \bigl(\frac{r-1}{r}+\delta\bigr)n$. Since $A\cup \{v\}$ induces a $K_{2r}$,
	there can be at most $(2r-1)\delta n$ neighbours of $v$ outside $J$. 
	The same argument applies to $w$ as well and thus we have
	\[
		|N(v)\sm J|+|N(w)\sm J|\le (4r-2)\delta n\,.
	\]
	Putting everything together one obtains
	\begin{align*}
		 |N(v)\sm N(w)| &\le |N(v)\sm J|+|J\sm N(w)|\le |N(v)\sm J|+|N(w)\sm J| +|J|-d(w) \\
		 &\le (4r-2)\delta n + \bigl(\tfrac{r-1}{r}+\delta\bigr)n-d(w)\,,
	\end{align*}
	which is slightly stronger than the estimate~\eqref{eq:vw}.
	
	We are now ready to verify~\ref{it:iii}. Let $xy$ denote an arbitrary edge of $G$ and 
	suppose that $N(x)\cup N(y)\ne V(G)$. This means that there exists a further vertex $z$ 
	with $xz, yz\not\in E(G)$ and two applications of~\eqref{eq:vw} reveal
	\begin{align*}
		|N(x)\cap N(y)| &\ge |N(z)|-|N(z)\sm N(x)|-|N(z)\sm N(y)| \\
		&\ge d(x)+d(y)+d(z)-2\bigl(\tfrac{r-1}{r}+4r\delta\bigr)n \\
		& \ge d(x)+d(y)-\bigl(\tfrac{r-1}{r}+8r\delta\bigr)n\,,
	\end{align*}
	as desired.  
\end{proof}

\subsection{Counting edges}

Next we prove a version of our intended result for graphs satisfying a minimum degree 
condition. 
 
\begin{prop} \label{prop-odd-mindeg}
	Suppose that $r$ is a positive integer and $0<\delta<\frac{1}{17r^3}$. If $G$
	is a $K_{2r+1}$-free graph with $n$ vertices, $\delta(G)\ge \frac{r-1}{r}n$, 
	and $\alpha(G)<\delta n$, then $e(G)\le \bigl(\frac{r-1}{r}+\delta\bigr)\frac{n^2}{2}$. 
\end{prop}

\begin{proof}
	Adding further edges to $G$ may create a $K_{2r+1}$ but cannot destroy any of the 
	other assumptions and thus we may assume that $G$ is actually an edge-maximal 
	$K_{2r+1}$-free graph. Let us colour an edge $xy$ of $G$ {\it red} if 
	$|N(x)\cap N(y)|<\frac{r-1}{r}\bigl(d(x)+d(y)\bigr)-\frac{r-1}{r+1}n$ and {\it green}
	otherwise. In view of Corollary~\ref{cor:red} we know that at most $\frac{r-1}{2r}n^2$ 
	edges of $G$ are red and thus it suffices to prove that at most $\delta n^2/2$ edges 
	of $G$ are green.
	If this failed, then some vertex $x$ would have more than $\delta n$ green neighbours
	and, consequently, there would exist a triangle~$xyz$ such that $xy$ and $xz$ are green,
	while the colour of $yz$ is unknown. The definition of~$xy$ being green leads to
	\[
		|N(x)\cup N(y)|=d(x)+d(y)-|N(x)\cap N(y)|
		\le \tfrac1r\bigl(d(x)+d(y)\bigr)+\tfrac{r-1}{r+1}n\,,
	\]
	by Lemma~\ref{lem:max}\ref{it:i} it follows that 
	\[
		 |N(x)\cup N(y)|\le \bigl( \tfrac{2(r-1)}{r^2}+\tfrac{r-1}{r+1} +4\delta \bigr) n<n\,,
	\]
	and hence Lemma~\ref{lem:max}\ref{it:iii} yields  
	\[
		|N(x)\cap N(y)|\ge d(x)+d(y)-\bigl(\tfrac{r-1}{r}+8r\delta\bigr)n\,.
	\]
	Proceeding similarly with the green edge $xz$ one shows 
	\[
		|N(x)\cap N(z)|\ge d(x)+d(z)-\bigl(\tfrac{r-1}{r}+8r\delta\bigr)n\,,
	\]
	so that altogether 
	\begin{align*}
		|N(x)\cap N(y)\cap N(z)| &\ge |N(x)\cap N(y)|+|N(x)\cap N(z)|-|N(x)| \\
		&\ge d(x)+d(y)+d(z)-2\bigl(\tfrac{r-1}{r}+8r\delta\bigr)n \\
		& \ge \bigl(\tfrac{r-1}{r}-16r\delta\bigr)n \ge\bigl(\tfrac{2r-3}{2r}+r\delta\bigr)n\,.
	\end{align*}

	Now applying Lemma~\ref{lem:max}\ref{it:ii} to the set $Q=N(x)\cap N(y)\cap N(z)$
	we find a $K_{2r+1}$ in $G$, which is absurd.  
\end{proof}

\begin{proof}[Proof of Theorem~\ref{thm:rt-odd}]
	For technical reasons it is more convenient to prove a slightly weaker upper bound first, 
	namely 
	\begin{equation}\label{eq:neu}
		e(G)\le \frac{r-1}r\cdot\frac{n^2+n}2+\frac{\delta n^2}2\,.
	\end{equation}

	Arguing indirectly, let $G$ be a $K_{2r+1}$-free graph on $n$ vertices 
	with $\alpha(G)<\delta n$ violating~\eqref{eq:neu}. Let $X\subseteq V(G)$
	be minimal with the property
	\begin{equation}\label{eq:XX}
		e(X) > \frac{r-1}r\cdot\frac{|X|^2+|X|}2+\frac{\delta n^2}2\,,
	\end{equation}
	let $G'$ be the subgraph of $G$ induced by $X$, and write $n'=|X|$.
	As $X$ cannot be empty, we may define $\delta'=\delta n/n'$. Now we would like 
	to apply Proposition~\ref{prop-odd-mindeg} to $G'$ and $\delta'$.
	 
	Notice that the trivial bound $e(X)\le |X|^2/2$ and~\eqref{eq:XX} lead 
	to $n'^2/r >\delta n^2$, whence $r(\delta')^2<\delta<\frac{1}{289r^5}$. 
	Thus we have indeed $\delta'<\frac 1{17r^3}$. Moreover, for every $x\in X$
	the minimality of $X$ yields
	\[
		e(X\sm\{x\}) \le \frac{r-1}r\cdot\frac{|X|^2-|X|}2+\frac{\delta n^2}2
	\]
	and, therefore, $d(x)=e(X)-e(X\sm\{x\})>\frac{r-1}r|X|$. As $x\in X$ was arbitrary, 
	this shows that~$X$ satisfies the required minimum degree condition. Finally,
	${\alpha(G')\le \alpha(G)=\delta n=\delta' |X|}$ is clear. 
	
	So Proposition~\ref{prop-odd-mindeg} discloses  
	\[
		e(X) \le \frac{r-1}r\cdot\frac{(n')^2}2+\frac{\delta'n'\cdot n'}2    
	 		< \frac{r-1}r\cdot\frac{|X|^2+|X|}2+\frac{\delta n\cdot n}2\,,
	\]
	contrary to~\eqref{eq:XX}. Thereby our weaker estimate~\eqref{eq:neu} is proved.
	
	Returning to the proof of Theorem~\ref{thm:rt-odd} itself we consider any graph 
	$G$ as described there. For every $t\in \NN$ let $G_t$ be the $t$-blow up of $G$, 
	i.e., the graph obtained from $G$ upon replacing every vertex by an independent set 
	consisting of $t$ new vertices. Of course $G_t$ is still $K_{2r+1}$-free and due to
	$\alpha(G_t)=t\alpha(G)<\delta |G_t|$ we may apply~\eqref{eq:neu} to $G_t$, thus learning
	\[
		e(G)=\frac{e(G_t)}{t^2}\le \frac{r-1}r\cdot\frac{n^2+n/t}2+\frac{\delta n^2}2\,.
	\]
	As $t\longrightarrow\infty$ this yields indeed
	$e(G)\le \bigl(\frac{r-1}{r}+\delta\bigr)\frac{n^2}{2}$.
\end{proof}

\section{Even cliques: Overview} \label{sec:even-intro}

The entire remainder of this article is concerned with the proof of Theorem~\ref{thm:main} 
and in the present section we would like to give an informal discussion of the strategy  
we shall pursue in the sequel. 

As in the case of odd cliques the first observation is that it suffices to focus on graphs 
satisfying an appropriate minimum degree condition, which is this time going to be 
$\delta(G)\ge\frac{3r-5}{3r-2}n$.
Besides, by making further sacrifices as to the eventual value of $\delta_*$, we can always 
assume that $n$ is sufficiently large. For these reasons, the main work goes into the proof 
of Proposition~\ref{thm:mindeg} below.

So let us suppose we have a sufficiently large $K_{2r}$-free graph $G$ 
with $\delta(G)\ge\frac{3r-5}{3r-2}n$ and $\alpha(G)<\delta n$, where $\delta$
is extremely small. Our task is to prove the upper bound 
$e(G)\le \bigl( \tfrac{3r-5}{3r-2}+\delta-\delta^2\bigr)\frac{n^2}{2}$
on the number of its edges.

The argument starts similar to the proof 
of~\eqref{eq:EHSS} given in~\cite{EHSS}. That is we apply Szemer\'edi's regularity lemma
and try to find one of several configurations in the regular partition, each of which 
would allow us to embed a $K_{2r}$. In~\cite{EHSS} this is done by applying some 
Tur\'an theoretic result to the reduced graph (see~\cite{EHSS}*{Lemma 3.3}) and 
the assumed absence of these configurations leads to an upper bound of the form
$e(G)\le \bigl( \tfrac{3r-5}{3r-2}+\delta'\bigr)\frac{n^2}{2}$ with 
$\delta'\to 0$ as $\delta\to 0$.

However, since for a given $\delta$ we are aiming at a somewhat better estimate 
on $e(G)$ than~\cite{EHSS} does, it may happen to us that this argument does not 
lead to immediate success. 
Yet there is still something we can do in order to proceed. Namely, we can prove a stability 
version of~\cite{EHSS}*{Lemma~3.3}, apply it to the reduced graph, and transfer the 
information thus obtained back to the original graph. In this manner, it can be shown that,
in an approximate sense, our graph $G$ does almost look like the extremal graph described 
in the proof of Proposition~\ref{prop:24}. Specifically, we find a partition 
\begin{equation}\label{eq:Apart}
	V(G)=A_1\dcup \ldots\dcup A_r
\end{equation}
such that each partition class spans at most $o(n^2)$ edges and the edge density between 
$A_1$ and~$A_2$ is, in a hereditary sense, at most $\frac 12+o(1)$ (see 
Proposition~\ref{prop:eta} below for a precise statement). Utilising the lower bound
$e(G)\ge \frac{3r-5}{3r-2}\cdot \frac{n^2}{2}$, which follows from the minimum degree 
assumption, one can prove that these two conditions imply that the partition classes 
$A_1, \ldots, A_r$ have roughly the expected sizes and that, as long as $\{i, j\}\ne \{1, 2\}$,
almost all possible edges between $A_i$ and $A_j$ are present in $G$ 
(see Fact~\ref{fakt:more} below).    
 
When one applies Proposition~\ref{prop:eta} to the essentially 
extremal graph constructed above, one ends up getting 
a partition which is to some extent similar to~\eqref{eq:Vi}, but it does not necessarily 
agree with it. More precisely, one could show that, perhaps after an appropriate permutation 
of the indices, one has $\sum_{i=1}^r |A_i\bigtriangleup V_i|=o(n)$. But the constant 
implied in the $o$-notation here could be extremely large in comparison to $\delta$ 
and thus it seems desirable to produce a better partition before one starts deriving the 
asymptotically optimal upper bound on~$e(G)$.  

Constructing such an improved partition is the subject of Subsection~\ref{subsec:local}.
Its main result, Proposition~\ref{prop:eps}, tells us that the graph $G$ under 
consideration admits a so-called {\it exact partition} $V(G)=B_1\dcup \ldots\dcup B_r$ 
satisfying a long list of properties enumerated in Definition~\ref{dfn:reps}. 
These conditions are rather restrictive and it might be helpful to imagine that, 
up to a relabeling of the indices,~\eqref{eq:Vi} is the only exact partition of 
the extremal graph. The proof of Proposition~\ref{prop:eps} starts from the 
partition~\eqref{eq:Apart} and is based on an iterative procedure that moves vertices around
that do not properly fit into the partition class they currently belong to.

Finally, in Section~\ref{sec:refine} we address the question how the knowledge
of an exact partition allows us to prove an upper bound on $e(G)$ (see 
Proposition~\ref{prop:exact}). The starting point there is the equation 
\[
	2e(G)=\sum_{i=1}^r e(B_i, V)\,.
\]
It turns out that one can separately show upper bounds for each of these terms, namely 
\begin{equation}\label{eq:24}
	e(B_i, V)\le |B_i|(n-|B_1|-|B_2|)+\tfrac12 |B_1||B_2|
				+\tfrac 12\delta n (|B_1|+|B_2|)-\tfrac 12\delta^2n^2
\end{equation}
for $i=1, 2$ (see Claim~\ref{clm:eB1}  below) and 	
\begin{equation}\label{eq:25}
	e(B_i, V)\le |B_i|(n-|B_i|)+\delta n |B_i| 
\end{equation}
for $i=3, \ldots, r$ (see Claim~\ref{clm:eBi}). 
By adding these estimates and optimising over $\sum_{i=1}^r|B_i|=n$ one obtains 
the desired bound $e(G)\le \bigl( \tfrac{3r-5}{3r-2}+\delta-\delta^2\bigr)\frac{n^2}{2}$.

Notice that there are two cases in which~\eqref{eq:25} is rather easy. First, if 
$B_i$ happens to be triangle-free, we get $e(B_i)\le \frac12\delta n|B_i|$ from 
$\alpha(G)<\delta n$ and by adding the trivial upper bound $e(B_i, V\sm B_i)\le |B_i|(n-|B_i|)$
the claim follows. Second, if it happens that $B_i$ misses at least~$2\eps n^2$ edges 
to $V\sm B_i$ for an appropriate (absolute) constant $\eps>0$, then the weaker 
bound $e(B_i)\le \eps n^2$, which exact partitions always satisfy, 
is enough to deduce~\eqref{eq:25}. The general 
argument is a superposition of these two cases. That is, we will define a partition
of $B_i$ into a triangle-free part $B^+_i$ to which the first argument applies and another part
$B^-_i$ that misses sufficiently many edges to $V\sm B_i$ to make the second approach useful.

The estimate~\eqref{eq:24} is much harder. Let us focus here on the case $r=2$ and $i=1$, 
in which many of the difficulties are already visible. To keep this overview simple
we will also assume that every vertex in $B_1$ sends at least $\tfrac 12|B_2|-\tfrac 1{60}n$
edges to $B_2$. 
Recall that in the extremal example there is a set $S\subseteq B_1$ of size close 
to $\delta n$ whose members are complete to $B_2$, whilst each vertex in $B_1\sm S$ 
sends a little bit less than $\frac12(|B_2|+\delta n)$
edges to $B_2$. Moreover, there is only a negligible number of edges within $B_1$.  
To prove~\eqref{eq:24}, we can {\it define} $S$ to be set 
of all $v\in B_1$ that send at least, say, $\frac{7}{16}n$ edges to $B_2$ (recall that 
$|B_2|\approx \frac12 n$). But even if we knew that $|S|\approx \delta n$ and were able to deal 
with $e(B_1, B_2)$, it would still be hard to control $e(B_1)$. 
The key to this problem is to prove that, as in the extremal example, there are $(i)$ 
no edges at all from $S$ to $B_1$ (see Fact~\ref{fact:S} below) and $(ii)$ no short 
odd cycles in~$B_1$ (see Fact~\ref{fact:P}). 
The latter fact helps us in the light of Lemma~\ref{lem:C7} below. 
  
Needless to say, many arguments occurring in this proof are inspired by~\cite{FLZ15}. 
But even for $r=2$ several new ideas are needed for going beyond~\eqref{eq:FLZ}.

\section{Coarse structure}
\label{sec:coarse}

Now we start to analyse the structure of $K_{2r}$-free graphs with huge minimum degree
but without linear independent sets. The main result  we shall obtain in this section
reads as follows.

\begin{prop} \label{prop:eta}
	Given an integer $r\ge 2$ and a real $\eta>0$ there exist $n_0\in\NN$ and $\delta>0$
	such that for every $K_{2r}$-free graph $G$ on $n\ge n_0$ vertices with 
	$\alpha(G)<\delta n$ and $\delta(G)\ge \frac{3r-5}{3r-2}n$ 
	there is a partition 
		\[
		V(G)=A_1\dcup A_2\dcup \ldots \dcup A_{r}
	\]
		with the following properties:
	\begin{enumerate}[label=\rmlabel]
	\item\label{it:eta1} $e(A_i)\le \eta n^2$ for all $i\in [r]$;
	\item\label{it:eta2} if $X_1\subseteq A_1$ and $X_2\subseteq A_2$, then 
		$e(X_1, X_2)\le \frac12 |X_1||X_2|+\eta n^2$.
	\end{enumerate}
\end{prop}

This will be shown by means of Szemer\'edi's famous regularity lemma~\cite{Sz78} 
and we commence by introducing some terminology.
Given a graph~$G$ and two nonempty disjoint sets $A, B\subseteq V(G)$ we say 
for two real numbers~$\delta>0$ and $d\in [0, 1]$
that the pair $(A, B)$ is {\it $(\delta, d)$-quasirandom} if for all $X\subseteq A$ 
and $Y\subseteq B$ the estimate $\big|e(X, Y)-d|X||Y|\big|\le \delta |A||B|$ holds. 
If we just say that the pair $(A, B)$ is {\it $\delta$-quasirandom} we mean that it 
happens to be $(\delta, d)$-quasirandom for $d=e(A, B)/|A||B|$.

\begin{thm}[Szemer\'edi's regularity lemma] 
	Given $\xi>0$ and $t_0\in\NN$ there exists an integer $T_0$ such that every graph $G$
	on $n\ge t_0$ vertices admits a partition 
		\begin{equation}\label{eq:Sz}
		V(G)=V_0\dcup V_1\dcup \ldots\dcup V_t
	\end{equation}
		of its vertex set such that 
	\begin{enumerate}
		\item[$\bullet$] $t\in [t_0, T_0]$, $|V_0|\le \xi |V(G)|$, and $|V_1|=\ldots=|V_t|>0$, 
		\item[$\bullet$] and for every $i\in [t]$ the set
						\[
				\bigl\{j\in [t]\sm \{i\}\colon (V_i, V_j) \text{ is not $\xi$-quasirandom}\bigr\}
			\]
						has size at most $\xi t$.
	\end{enumerate}
\end{thm}

In the literature one often finds other versions of the regularity lemma, where instead of the 
second bullet above it is just demanded that at most $\xi t^2$ pairs $(V_i, V_j)$ 
with distinct $i, j\in [t]$ fail to be $\xi$-quasirandom. Applying such a regularity 
lemma to appropriate constants $\xi'\ll \xi$ and $t_0'\gg \max(t_0, \xi^{-1})$ and relocating partition 
classes with many irregular partners to $V_0$ one can obtain the version stated here; 
this argument has been used before by \L uczak~\cite{Lu06}, who explains it in more detail. 

Next we deal with certain configurations in regular partitions of graphs with small 
independence number which allow us to build cliques. The lemma that follows is implicit 
in~\cite{EHSS}*{Section 4} but for reasons of self-containment we shall supply its short 
proof. In its formulation we work with a one-sided version of quasirandomness that 
is enough for our purposes: If $G$ is a graph, a pair $(A, B)$ of disjoint subsets 
of $V(G)$ is said to be {\it $(\delta, d)$-dense} for $\delta>0$ and $d\in [0, 1]$, 
if for all $X\subseteq A$ and $Y\subseteq B$ we have $e(X, Y)\ge d |X||Y|-\delta |A||B|$. 

\begin{lemma} \label{lem:ab}
	Suppose that integers $a\ge b\ge 1$ as well as a real number $\theta\in(0, 1]$
	are given and set $\xi=\bigl(\frac{\theta^2}4\bigr)^{a-1}$, 
	$\delta=\bigl(\frac{\theta}2\bigr)^{a-1}$. Let $H$ be a 
	graph possessing a vertex partition 
	\[
		V(H)=V_1\dcup \ldots \dcup V_a
	\]
	into nonempty classes satisfying 
	\begin{enumerate}[label=\alabel]
		\item\label{it:aba} if $1\le i<j\le a$, then $(V_i, V_j)$ is 
			$(\xi, d_{ij})$-dense for some $d_{ij}\in [\theta, 1]$;
		\item\label{it:abb} if $1\le i<j\le b$, then $d_{ij}\ge \frac12+\theta$; 
		\item\label{it:abc} if $X\subseteq V_i$ and $|X|\ge \delta |V_i|$ 
			for some $i\in [a]$, then $X$ spans at least one edge in $H$.
	\end{enumerate}
	Then $H$ contains a clique of order $a+b$.
\end{lemma}

\begin{proof}
	We argue by induction on $a+b$. In the base case, $a=b=1$, we have $\delta=1$
	and by condition~\ref{it:abc} applied to $X=V_1$ there is indeed an edge in $H$.
	
	In the induction step we certainly have $a\ge 2$ and we assume first that $a>b$. 
	For every $i\in [a-1]$ the set 
		\[
		X(i)=\bigl\{v\in V_a\colon |N(v)\cap V_i|\le \tfrac\theta 2|V_i|\bigr\}
	\]
		cannot be very large, as condition~\ref{it:aba} yields
		\[
		\tfrac\theta 2 |V_i||X(i)|\ge e(V_i, X(i)) \ge \theta |V_i||X(i)|-\xi |V_i||V_a|\,.
	\]
	Together with $\xi\le \frac{\theta}{2a}$ 
	this leads to $|X(i)|\le  \frac{1}{a}|V_a|$.
	Now pick some ${v_*\in V_a\sm \bigcup_{i\in [a-1]}X(i)}$
	and set $V'_i=N(v_*)\cap V_i$ for $i=1, \ldots, a-1$.
	The definition of $X(i)$ gives 
	$|V'_i|\ge \frac{\theta}2 |V_i|$ for every $i\in [a-1]$ and, hence,
	the sets $V'_1, \ldots, V'_{a-1}$ 
	have the above properties~\ref{it:aba},~\ref{it:abb}, and~\ref{it:abc}
	for $a-1$, $\tfrac{\xi}{\theta^2/4}$, and $\frac{\delta}{\theta/2}$ here 
	in place of $a$, $\xi$, and $\delta$ there. So by the induction hypothesis 
	the neighbourhood of $v_*$ contains a $K_{a+b-1}$, wherefore indeed 
	$K_{a+b}\subseteq H$.
	
	The case $a=b$ is similar, but instead of the sets $X(i)$ introduced above we 
	consider
		\[
		Y(i)=\bigl\{v\in V_a\colon 
			|N(v)\cap V_i|\le \bigl(\tfrac 12+\tfrac\theta 2\bigr)|V_i|\bigr\}
	\]
		for $i\in [a-1]$. Invoking condition~\ref{it:abb} one can show 
	$|Y(i)|\le \frac{1}{a}|V_a|$ in the same way as before and, hence, the set 
	$L=V_a\sm \bigcup_{i\in [a-1]}Y(i)$ satisfies $|L|\ge \frac 1a|V_a|\ge \delta |V_a|$.
	So by~\ref{it:abc} there is an edge $v_*w_*$ both of whose endvertices belong to $L$.
	Since $|N(v_*)\cap N(w_*)\cap V_i|\ge \theta |V_i|$ holds for each $i\in [a-1]$,
	the induction hypothesis allows us to find a $K_{a+b-2}$ in the common neighbourhood
	of $v_*w_*$ and again we obtain $K_{a+b}\subseteq H$. 
\end{proof}

Suppose now that the regularity lemma has been applied, with a sufficiently small 
accuracy parameter $\xi$, to some graph $G$ of small independence number, meaning
that for some large integer $t$ we have a partition of $V(G)$ such as~\eqref{eq:Sz}. 
When one now attempts to find a $K_{2r}$ in $G$ by means of Lemma~\ref{lem:ab}, it only
matters which of the quasirandom pairs~$(V_i, V_j)$ have their densities, for an 
appropriate $\theta>0$, in the interval $\bigl[\theta, \frac12+\theta\bigr)$ or even in  
$\bigl[\frac12+\theta, 1\bigr]$. We shall encode such information by the use of coloured
edges in the reduced graph, with green edges corresponding to pairs that are either irregular
or too sparse to be useful, and blue (or red) edges corresponding to quasirandom pairs 
of medium (or large) density.  

Let us say that a {\it coloured graph} is a complete graph all of whose edges
have been coloured red, blue, or green. Associated with any coloured graph $G$,
say with vertex set $V$, we have its so-called {\it weight function} 
$w\colon V^2\longrightarrow \{0, 1, 2\}$ defined by 
\[
	w(x, y)=\begin{cases}
		0 & \text{ if $x=y$ or $xy$ is green,} \cr
		1 & \text{ if $xy$ is blue,} \cr
		2 & \text{ if $xy$ is red}
		\end{cases}
\]
for all $x, y\in V$. We will often identify $G$ with the pair $(V, w)$. The {\it degree}
of a vertex $x$ of a coloured graph $G=(V, w)$ is defined to be the sum
\[
	d(x)=\sum_{y\in V} w(x, y)
\]
and by $e(G)$ we mean half of the sum of the degrees $d(x)$ as $x$ varies over $V$.

Two coloured graphs are said to be {\it isomorphic} if there is a colour-preserving bijection 
between their vertex sets. A coloured graph $(V', w')$ is a {\it subgraph} of a coloured 
graph~$(V, w)$ if $V'\subseteq V$ and, additionally, $w'(x, y)\le w(x, y)$ holds 
for all $x, y\in V'$. 

Next, we come to the coloured graphs which are relevant in connection with Lemma~\ref{lem:ab}.
For integers $a\ge b\ge 1$ the coloured graph on $a$ vertices without green edges 
whose red edges form a clique of order $b$ will be denoted by $G_{a+b, b}$. 
For every integer $r\ge 2$ we set $\ccF_{2r}=\{G_{2r, 1}, \ldots, G_{2r, r}\}$. 
A coloured graph is said to be $\ccF_{2r}$-free if none of its subgraphs 
is isomorphic to a member of $\ccF_{2r}$.

In their proof of~\eqref{eq:EHSS}, Erd\H{o}s, Hajnal, Szemer\'edi, and S\'os
use a lemma saying that every $\ccF_{2r}$-free coloured graph on $n$ vertices 
satisfies $e(G)\le \frac{3r-5}{3r-2}n^2$ (see~\cite{EHSS}*{Lemma~3.3}). For the proof
of Proposition~\ref{prop:eta} we will use a stability version of this lemma. There 
are various such statements, a rather strong one being the following.

\begin{prop}\label{prop:clara} 
	Suppose that $r\ge 2$ and that $G$ is a $\ccF_{2r}$-free coloured graph on $n$ 
	vertices with $\delta(G)>\frac{14r-24}{7r-5}n$. Then there is a partition 
	$V(G)=W_1\dcup\ldots\dcup W_r$ such that all edges within the partition 
	classes are green and there are no red edges between $W_1$ and $W_2$.
\end{prop}

A somewhat lengthy proof of this result is given in~\cite{LR-b}.
For the purposes of the present work, however, it suffices to know only the weaker statement 
that follows. To keep this article as self-contained as possible, we 
will supply a quick sketch of its proof below. 

\begin{prop}\label{prop:silly}
	Let $r\ge 2$ and let $\alpha>0$ be sufficiently small. Then every 
	$\ccF_{2r}$-free coloured graph $G$ on $n$ vertices with 
	$\delta(G)\ge \frac{2(3r-5)-\alpha}{3r-2}n$ admits a partition 
		\[
		V(G)=W_0\dcup W_1\dcup\ldots\dcup W_r
	\]
		of its vertex set such that 
	$|W_0|\le \alpha n$, all edges within the classes $W_1, \ldots, W_r$ are green,
	and no edge from $W_1$ to $W_2$ is red.
\end{prop} 	  

We prepare the proof of this proposition by the following variant of~\cite{EHSS}*{Lemma~3.3},
which can be proved in the same way. Let $RK_{r-1}$ denote a red clique of order $r-1$
and set $\ccF^+_{2r}=\ccF_{2r}\cup\{RK_{r-1}\}$. 

\begin{lemma}\label{lem:Fr+}
	For $r\ge2$ every $\ccF^+_{2r}$-free coloured graph $G$
	on $n$ vertices satisfies 
	\[
		e(G)\le \tfrac{r-2}{r-1} n^2\,.
	\]
\end{lemma}

\begin{proof}
	The case $r=2$ is clear, for a $RK_1$ is nothing else than a vertex. 
	So suppose $r\ge 3$ from now on.
	As in~\cite{EHSS}, two consecutive applications of Zykov's symmetrisation method~\cite{Zy}
	show that we may assume that there are is partition $V(G)=A_1\dcup \ldots\dcup A_m$
	and that for each~$i\in [m]$ there is a partition $A_i=B_{i1}\dcup\ldots\dcup B_{ik_i}$
	such that   
	\begin{enumerate}[label=\rmlabel]
		\item\label{it:z1} for $i\in [m]$ and $j\in [k_i]$ all edges within $B_{ij}$ are green;
		\item\label{it:z2} if $i\in [m]$ and $j, j'\in [k_i]$ are distinct, then all edges 
			between $B_{ij}$ and $B_{ij'}$ are blue;
		\item\label{it:z3} and for distinct $i, i'\in [m]$ all edges between $A_i$ and $A_{i'}$
			are red.
	\end{enumerate}
	
	Since $G$ contains neither $RK_{r-1}$ nor $G_{2r, m}$, we have 
		\begin{equation}\label{eq:mki}
		1\le m\le r-2
		\qquad \text{ and } \qquad
		k_1+\ldots+k_m\le 2r-1-m\,.
	\end{equation}
		
	Set $\alpha_i=|A_i|/n$ for $i\in [m]$ and notice that $\sum_{i=1}^m\alpha_i=1$. 
	It is well known that~\ref{it:z1} and~\ref{it:z2}
	imply $e(A_i)\le \frac{k_i-1}{2k_i}|A_i|^2$ and thus it remains to prove 
		\[
		\sum_{1\le i\le m} \tfrac{k_i-1}{2k_i}\alpha_i^2
		+ 2\sum_{1\le i<j\le m} \alpha_i \alpha_j \le \tfrac{r-2}{r-1}\,.
	\]
		Subtracting this from $\bigl(\sum_{i=1}^m\alpha_i\bigr)^2=1$ we get 
		\[
		\sum_{i=1}^m \tfrac{k_i+1}{2k_i}\alpha_i^2\ge \tfrac{1}{r-1}\,.
	\]
		The Cauchy-Schwarz inequality yields 
		\[
		\sum_{i=1}^m \tfrac{k_i+1}{2k_i}\alpha_i^2 \cdot \sum_{i=1}^m \tfrac{2k_i}{k_i+1}
		\ge \bigl(\sum_{i=1}^m\alpha_i\bigr)^2=1
	\]
		and thus it suffices to show that 
		\begin{equation}\label{eq:ki}
		\sum_{i=1}^m \tfrac{k_i}{k_i+1} \le \tfrac{r-1}{2}\,.
	\end{equation}
		
	Since the estimate $\frac{k}{k+1}\le \frac{k+2}{6}$ holds for each positive 
	integer~$k$, it is enough to prove 
		\[
		\sum_{i=1}^m \tfrac{k_i+2}{6} \le \tfrac{r-1}{2}
	\]
		instead and in view of~\eqref{eq:mki} this is clear.	 
\end{proof}

\begin{proof}[Proof of Proposition~\ref{prop:silly}]
	Since $\frac{r-2}{r-1}<\frac{3r-5}{3r-2}$ and $\alpha\ll 1$, we may suppose 
	that $e(G)>\frac{r-2}{r-1}$.
	By Lemma~\ref{lem:Fr+} and the assumption that $G$ be $\ccF_{2r}$-free 
	it follows that $G$ contains a $RK_{r-1}$, say with vertex set 
	$K=\{v_1, v_3, \ldots v_r\}$. The minimum degree condition and $\alpha\ll 1$ yield 
		\[
		\sum_{x\in V(G)}\bigl(2r-2-d_K(x)\bigr)=\sum_{v\in K}\bigl(2n-d(v)\bigr)
		\le \tfrac{(6+\alpha)(r-1)}{3r-2}n<2n
	\]
		and, hence, there is a vertex $v_2\in V(G)$ with $2r-2-d_K(v_2)\le 1$. 
	As $G$ contains no $G_{2r, r}=RK_r$, it follows that $v_2$ has exactly one
	blue neighbour in $K$ and sends red edges to all other members of $K$. By
	symmetry we may suppose that $v_1v_2$ is blue. Set 
	\begin{enumerate}
		\item[$\bullet$] $L=\{v_1, \ldots, v_r\}=K\cup\{v_2\}$,
		\item[$\bullet$] $W_i=\bigl\{x\in V(G)\colon \text{ if $j\in [r]$, then 
			$w(x, v_j)=w(v_i, v_j)$}\bigr\}$ for $i=1, \ldots, r$,
		\item[$\bullet$] $W_0=V(G)\sm (W_1\cup\ldots\cup W_r)$, 
		\item[$\bullet$] and $q(x)=2(3r-2)-2\bigl(w(v_1, x)+w(v_2, x)\bigr)
				-3\bigl(w(v_3, x)+\ldots+w(v_r, x)\bigr)$ for every $x\in V(G)$.
	\end{enumerate}
	Notice that the sets $W_1, \ldots, W_r$ are mutually disjoint.
	Exploiting that $G$ contains neither $G_{2r, r}$ nor $G_{2r, r-1}$ one checks easily
	that
	\begin{enumerate}
		\item[$\bullet$] all edges within one of the partition classes $W_1, \ldots, W_r$ 
			are green
		\item[$\bullet$] no edge from $W_1$ to $W_2$ is red,
		\item[$\bullet$] $q(x)\ge 6$ for all $x\in V(G)$,
		\item[$\bullet$] and that equality holds in the previous bullet if and only if 
			$x\in W_1\dcup\ldots\dcup W_r$.
	\end{enumerate}
	It remains to show that $|W_0|\le \alpha n$. To this end we write
	\[
		|W_0|\le \sum_{x\in V(G)}(q(x)-6)=2(3r-5)n-2\bigl(d(v_1)+d(v_2)\bigr)+
			3\bigl(d(v_3)+\ldots+d(v_r)\bigr) 
	\]
	and apply the minimum degree condition again. 
\end{proof}

Finally, we show the main result of this section. 

\begin{proof}[Proof of Proposition~\ref{prop:eta}]
	Take appropriate constants
		\[
		\delta\ll T_0^{-1} \ll t_0^{-1}, \xi\ll \theta\ll \min(\eta, r^{-1})\,,
	\]
		where $T_0$ is obtained by applying the regularity lemma to $t_0$ and $\xi$, 
	and set $n_0=t_0$. Consider a $K_{2r}$-free graph $G$ on $n\ge n_0$ vertices with 
	$\alpha(G)<\delta n$ and $\delta(G)\ge \frac{3r-5}{3r-2}n$. The regularity 
	lemma yields for some integers $t\in [t_0, T_0]$ and $m\ge 1$ a partition
		\[
		V(G)=V_0\dcup V_1\dcup \ldots \dcup V_t
	\]
		such that $|V_0|\le \xi n$, $|V_1|=\ldots=|V_t|=m$, and for every $i\in [t]$
	all but at most $\xi t$ indices $j\in [t]\sm \{i\}$ have the property that 
	$(V_i, V_j)$ is $\xi$-quasirandom.
	
	Define a coloured graph $H$ with vertex set $[t]$ by declaring a pair $ij$ to 
	be {\it green} if $(V_i, V_j)$ either fails to be $\xi$-quasirandom or has a 
	density smaller than $\theta$, {\it blue} if $(V_i, V_j)$ is $\xi$-quasirandom 
	and has a density in $\bigl[\xi, \tfrac12+\xi\bigr)$, and {\it red} otherwise.
	
	As a consequence of Lemma~\ref{lem:ab}, $H$ is $\ccF_{2r}$-free. Next, we will show 
	that 
		\begin{equation}\label{eq:delH}
		\delta(H)\ge 2\bigl(\tfrac{3r-5}{3r-2}-3\theta\bigr)t\,.
	\end{equation}
		To verify this, we consider an arbitrary vertex $i$ of $H$ and denote
	the numbers of its blue and red neighbours by $a$ and $b$, respectively.
	The minimum degree condition on $G$ yields 
		\[
		\tfrac{3r-5}{3r-2}mn\le \sum_{j=0}^t e(V_i, V_j)\,.
	\]
		On the right side of this estimate, the term corresponding to 
	$j=0$ contributes at most~$\xi mn$, $j=i$ contributes at most $m^2$,
	and the irregular pairs contribute at most $\xi tm^2$. 
	Consequently we have 
		\[
		\bigl(\tfrac{3r-5}{3r-2}-\xi\bigr)mn\le m^2+\xi tm^2+t\theta m^2
		+a\bigl(\tfrac12+\theta)m^2+bm^2. 
	\]
		Using $n\ge mt$ and canceling $m^2$ we infer
		\[
		 \bigl(\tfrac{3r-5}{3r-2}-\xi\bigr) t\le (2\theta+\xi)t+1+\tfrac12 d_H(i)\,.
	\]
		So in view of $t\ge t_0\gg \theta^{-1}$ and $\xi\ll \theta$ we obtain
	$d_H(i)\ge 2\bigl(\tfrac{3r-5}{3r-2}-3\theta\bigr)t$, which proves~\eqref{eq:delH}.
	
	By Proposition~\ref{prop:silly} and $\theta\ll r^{-1}$ there exists a partition 
		\[
		[t]=W_0\dcup W_1\dcup\ldots\dcup W_r
	\]
		such that $|W_0|\le 18\theta rt$, all edges within $W_1, \ldots, W_t$
	are green, and no edge between $W_1$ and $W_2$ is red. 
	For $s\in [0, r]$ we define
		\[
		A^*_s=\bigcup_{i\in W_s}V_i\,. 
	\]
		Then $V(G)=V_0\dcup A^*_0\dcup A^*_1\dcup\ldots \dcup A^*_r$ is a partition 
	of $V(G)$ and 
		\[
		|V_0|+|A^*_0|\le \xi n+|W_0|m\le (\xi+18r\theta)n\le\tfrac 12\eta n\,.
	\]
		
	This means that if we manage to show 
	\begin{enumerate}[label=\alabel]
	\item\label{it:q1} $e(A^*_s)\le\frac 12 \eta n^2$ for all $s\in [r]$,
	\item\label{it:q2} and $e(X_1, X_2)\le \tfrac 12|X_1|X_2|+\frac 12 \eta n^2$ 
		for all $X_1\subseteq A^*_1$ and $X_2\subseteq A^*_2$,
	\end{enumerate}
	then the partition $V(G)=A_1\dcup \ldots\dcup A_r$ defined by 
	$A_1=V_0\dcup A^*_0\dcup A^*_1$ and $A_s=A^*_s$ for $s\in [2, r]$ 
	has both desired properties.
	
	To prove~\ref{it:q1} we start for a given $s\in [r]$ from the decomposition 
		\[
		e(A^*_s)=\sum_{i\in W_s} e(V_i) + \sum_{ij\in W_s^{(2)}} e(V_i, V_j)\,.
	\]
		Here, each of the at most $t$ terms in the first sum is at most $m^2/2$. 
	Besides, there are at most $\xi t^2/2$ terms corresponding to irregular 
	pairs in the second sum, and each of them amounts to no more than $m^2$.
	Finally, the remaining at most $t^2/2$ terms in the second sum correspond to 
	pairs whose density is at most $\theta$. Thus we obtain
		\[
		 e(A^*_s)\le \bigl(\tfrac 1{2t}+\tfrac\xi 2+\tfrac \theta 2\bigr)m^2t^2
	\]
		and due to $t\ge t_0$ and $mt\le n$ an appropriate choice of our constants 
	does indeed guarantee that $e(A^*_s)\le\frac 12 \eta n^2$.
	
	Similarly, the proof of~\ref{it:q2} employs 
		\[
		e(X_1, X_2)=\sum_{i\in W_1}\sum_{j\in W_2}e(V_i\cap X_1, V_j\cap X_2)\,.
	\]
		Again the contribution caused by irregular pairs is at most $\xi n^2/2$.
	The remaining terms correspond to $\xi$-quasirandom pairs, which owing
	to the absence of red edges from $W_1$ to~$W_2$ have density at most $\tfrac 12+\theta$.
	Consequently,
		\begin{align*}
		e(X_1, X_2) &\le \sum_{i\in W_1}\sum_{j\in W_2}
			\Bigl[\bigl(\tfrac 12+\theta\bigr)|V_i\cap X_1||V_j\cap X_2|+\xi |V_i||V_j|\Bigr] 
			+\tfrac 12 \xi n^2 \\
		&\le \bigl(\tfrac 12+\theta\bigr)|X_1||X_2|+ \xi t^2m^2 +\tfrac 12 \xi n^2 \\
		&\le \tfrac 12 |X_1||X_2| + \bigl(\theta+\tfrac 32\xi\bigr) n^2 
		\le \tfrac 12 |X_1||X_2| + \tfrac 12 \eta n^2
	\end{align*}
		and the proof of Proposition~\ref{prop:eta} is complete.
\end{proof}
 
\section{Exact partitions} 
\label{sec:local}

\subsection{More information} It turns out that the lower bound 
$e(G)\ge \tfrac{3r-5}{3r-2}\cdot \tfrac{n^2}2$, which follows from 
the minimum degree condition in Proposition~\ref{prop:eta}, gives 
us further information on the sizes of the vertex classes of the
partition obtained there and on the edge densities between these classes.
This happens due to the following elementary inequality.  

\begin{lemma} \label{lem:ai}
	If for $r\ge 2$ the real numbers $a_1, \ldots, a_r$ sum up to $1$,
	then 
		\[
		\sum_{1\le i<j\le r}a_ia_j-\tfrac12 a_1a_2\le \tfrac{3r-5}{2(3r-2)}\,.
	\]
		Moreover, if for some real $\rho\ge 0$ we have 
		\begin{equation}\label{eq:rho}
		 \sum_{1\le i<j\le r}a_ia_j-\tfrac12 a_1a_2\ge \tfrac{3r-5}{2(3r-2)}-\rho\,,
	\end{equation}
		then $\big|a_i-\frac{2}{3r-2}\big|\le 2\sqrt{\rho}$ for $i=1, 2$ and 
	$\big|a_i-\frac{3}{3r-2}\big|\le 2\sqrt{\rho}$ for $i=3, \ldots, r$.
\end{lemma}

\begin{proof}
	Define 
		\[
		\alpha_i=\begin{cases}	
					a_i-\frac{2}{3r-2} & \text{if $i=1, 2$} \\
					a_i-\frac{3}{3r-2} & \text{if $i=3, \ldots, r$}
				\end{cases}
	\]
		and observe that 
		\[
		\sum_{i=1}^r\alpha_i^2+\alpha_1\alpha_2
		=\sum_{i=1}^r a_i^2+a_1a_2-\sum_{i=1}^r\frac{6a_i}{3r-2}+\frac{4\cdot 3+9(r-2)}{(3r-2)^2}
		=\sum_{i=1}^r a_i^2+a_1a_2-\frac{3}{3r-2}\,.
	\]
		Due to $\bigl(\sum_{i=1}^r a_i\bigr)^2=1$ this rewrites as
		\[
		\tfrac12 \alpha_1^2+\tfrac12 \alpha_2^2 +\tfrac 12(\alpha_1+\alpha_2)^2
			+\sum_{i=3}^{r}\alpha_i^2
		\le \tfrac{3r-5}{3r-2}-\Bigl(2\sum_{i<j}a_ia_j-a_1a_2\Bigr)\,,
	\]
	which establishes the first part of our claim. Moreover, if~\eqref{eq:rho}
	holds for some $\rho\ge 0$ we obtain
		\[
		\tfrac12 \alpha_1^2+\tfrac12 \alpha_2^2 +\tfrac 12(\alpha_1+\alpha_2)^2
			+\sum_{i=3}^{r}\alpha_i^2
		\le 2\rho\,,
	\]
		whence $|\alpha_i|\le 2\sqrt{\rho}$ holds for all $i\in [r]$.
\end{proof}

With this lemma at hand we may prove the following estimates. 

\begin{fakt}\label{fakt:more}
	Suppose that a graph $G$ and the partition 
		\[
		V(G)=A_1\dcup \ldots\dcup A_r
	\]
		are as described and obtained in Proposition~\ref{prop:eta}. Then 
	\begin{enumerate}
		\item[$\bullet$] $\big||A_i|-\frac{2n}{3r-2}\big|\le 2\sqrt{(r+1)\eta}\cdot n$ 
			for $i=1, 2$,
		\item[$\bullet$] $\big||A_i|-\frac{3n}{3r-2}\big|\le 2\sqrt{(r+1)\eta}\cdot n$ for 
			$i=3, \ldots, r$,
		\item[$\bullet$] $e(A_1, A_2)\ge \tfrac12 |A_i||A_j|-r\eta n^2$,
		\item[$\bullet$] and $e(A_i, A_j)\ge |A_i||A_j|-(r+1)\eta n^2$ whenever $1\le i<j\le n$
			and $(i, j)\ne (1, 2)$.
	\end{enumerate}
\end{fakt}

\begin{proof}
	The minimum degree condition $\delta(G)\ge\frac{3r-5}{3r-2}n$ yields 
	$e(G)\ge \frac{3r-5}{3r-2}\cdot \frac{n^2}{2}$ and due to 
	Proposition~\ref{prop:eta}\ref{it:eta1} it follows that
		\begin{align}\label{eq:long}
		\notag \bigl(\tfrac{3r-5}{2(3r-2)}&-(r+1)\eta\bigr)n^2+
		\sum_{\substack{1\le i<j\le r \\ (i, j)\ne (1, 2)}} \bigl[|A_i||A_j|-e(A_i, A_j)\bigr]
		+\bigl[\tfrac 12|A_1||A_2|+\eta n^2-e(A_1, A_2)\bigr] \\
		&\le \sum_{1\le i<j\le r} |A_i||A_j| - \tfrac 12|A_1||A_2|\,.
	\end{align}
		
	The square brackets on the left side being positive we deduce 
		\[
		\bigl(\tfrac{3r-5}{2(3r-2)}-(r+1)\eta\bigr)n^2\le 
		\sum_{1\le i<j\le r} |A_i||A_j| - \tfrac 12|A_1||A_2|
	\]
		and the case $\rho=(r+1)\eta$ of Lemma~\ref{lem:ai} leads to the first two bullets.
	
	Furthermore, Lemma~\ref{lem:ai} provides an upper bound of 
	$\tfrac{3r-5}{3r-2}\cdot \frac{n^2}2$ on the right side of~\eqref{eq:long}.
	Therefore we have
		\[
		\sum_{\substack{1\le i<j\le r \\ (i, j)\ne (1, 2)}} \bigl[|A_i||A_j|-e(A_i, A_j)\bigr]
		+\bigl[\tfrac 12|A_1||A_2|+\eta n^2-e(A_1, A_2)\bigr]\le (r+1)\eta n^2\,.
	\]
		and the last two bullets follow as well.
\end{proof}
	
\subsection{Local minimum degree} \label{subsec:local}
Along the way leading from the partition provided by Proposition~\ref{prop:eta} to our 
main theorem we will need to make further efficient uses of the assumption 
$K_{2r}\not\subseteq G$. It should be clear that {\it building} a $K_{2r}$ in $G$
would be easier if we knew that certain minimum degree conditions hold between
the partition classes and the goal of this section is to enforce several such conditions
by moving a few vertices violating them to other classes into which they fit better.
For later reference we include the somewhat lengthy list of properties that we shall obtain
into a definition.

\enlargethispage{2em}

\begin{dfn}\label{dfn:reps}
	Let an integer $r\ge 2$, a real $\eps>0$, an $n$-vertex graph $G$, and a partition 
		\begin{equation*}
		V(G)=B_1\dcup\ldots\dcup B_r
  	\end{equation*}
		be given. Set $d_i(v)=d_{B_i}(v)$ for all $v\in V(G)$ and $i\in [r]$.
	We say that the above partition is {\it $(r, \eps)$-exact}
	if the following conditions hold.
	\begin{enumerate}[label=\glabel]
	\item\label{it:r1} For $i=1,2$ one has $\big||B_i|-\frac{2n}{3r-2}\big|\le \eps n$.
	\item\label{it:r2} For $i=3, \ldots, r$ one has $\big||B_i|-\frac{3n}{3r-2}\big|\le \eps n$. 
	\item\label{it:r3} If $i\in [r]$, then $e(B_i)\le \eps n^2$.  
	\item\label{it:r4} If $X_1\subseteq B_1$ and $X_2\subseteq B_2$, then 
		 $\big|e(X_1, X_2)-\tfrac12 |X_1||X_2|\big|\le \eps n^2$.
	\item\label{it:r5} If $1\le i<j\le r$ and $(i, j)\ne (1, 2)$, 
			then $e(B_i, B_j)\ge |B_i||B_j|-\eps n^2$.
	\item\label{it:r6} If $\{i, j\}=\{1, 2\}$ and $v\in B_i$, then 
		$d_j(v)\ge\frac{1/3}{3r-2}n$.
	\item\label{it:r7} If $i\in \{1, 2\}$, $j\in [3, r]$, and $v\in B_i$, then  
		$d_j(v)\ge\frac{5/3}{3r-2}n$.
	\item\label{it:r8} If $i\in [3, r]$, $j\in \{1, 2\}$, and $v\in B_i$,
		then $d_j(v)\ge\frac{1/5}{3r-2}n$.
	\item\label{it:r9} If $i, j\in [3, r]$ are distinct and $v\in B_i$,
		then $d_j(v)\ge\frac{1}{3r-2}n$.

	\end{enumerate}
\end{dfn}

The main result of this subsection is the following.

\begin{prop} \label{prop:eps}
	For every $r\ge 2$ and $\eps>0$ there exist $n_0\in\NN$ and $\delta>0$
	such that every $K_{2r}$-free graph $G$ on $n\ge n_0$ vertices,
 	with $\delta(G)\ge\frac{3r-5}{3r-2}n$ and $\alpha(G)<\delta n$
	has an $(r, \eps)$-exact partition.
\end{prop}

\begin{proof}
	By monotonicity we may assume that $\eps$ is sufficiently small so that 
	all estimates to be performed below will hold. We commence be choosing a sufficiently
	small $\eta\ll \eps$. With this number $\eta$ we appeal to Proposition~\ref{prop:eta} and it 
	answers with an integer $n_0\in\NN$ and with some $\delta>0$. We claim that these 
	two constants have the desired properties. 
	
	Let any $K_{2r}$-free graph $G$ on $n\ge n_0$ vertices with 
	$\alpha(G)<\delta n$ and $\delta(G)\ge \frac{3r-5}{3r-2}n$ 
	be given and take a partition 
		\begin{equation} \label{eq:A0}
		V(G)=A^0_1\dcup A^0_2\dcup \ldots \dcup A^0_{r}
	\end{equation}
		such that	
	\begin{enumerate}[label=\rmlabel]
	\item\label{it:a1} $e(A^0_i)\le \eta n^2$ for all $i\in [r]$;
	\item\label{it:a2} if $X_1\subseteq A^0_1$ and $X_2\subseteq A^0_2$, 
		then $e(X_1, X_2)\le \frac12 |X_1||X_2|+\eta n^2$.
	\end{enumerate} 
	
	\noindent Due to Fact~\ref{fakt:more} and $\eta\ll\eps$ we may suppose moreover that 
	
	\begin{enumerate}[label=\rmlabel, resume] 
	\item\label{it:a3} for $i=1,2$ we have 
		$\big||A^0_i|-\tfrac{2n}{3r-2}\big|\le\frac 12\eps n$;
	\item\label{it:a4} for $i=3, \ldots, r$ we have 
		$\big||A^0_i|-\tfrac{3n}{3r-2}\big|\le\frac 12\eps n$;
	\item\label{it:a5} $e(A^0_1, A^0_2)\ge \tfrac 12|A^0_1||A^0_2|-\tfrac 14 \eps n^2$;
	\item\label{it:a6} and that $e(A^0_i, A^0_i)\ge |A^0_i||A^0_j|-\tfrac 12 \eps n^2$
		whenever $1\le i<j\le r$ and $(i, j)\ne (1, 2)$.
	\end{enumerate}

	We need to define an $(r, \eps)$-exact partition of $G$. 	
	To this end we perform a recursive procedure, in the course of which 
	a sequence of partitions of~$V(G)$ into $r$ parts is constructed.
	The starting point is~\eqref{eq:A0}. In each step
	only one vertex is moved from one vertex class to another one, while all other 
	vertices stay in the partition class they have belonged to before. Let 
		\[
		 V(G)=A^s_1\dcup A^s_2\dcup \ldots \dcup A^s_{r}
	\]
		be the partition that we have after $s$ steps and put
		\[
		\Omega_s=6e(A^s_1)+6e(A^s_2)+\sum_{i=3}^r e(A^s_i)\,.
	\]
		When the $s^{\text{th}}$ step is to carried out, we ensure that 
		\begin{equation}\label{eq:Omega}
		\Omega_s\le\Omega_{s-1}-\tfrac{1/4}{3r-2}n
	\end{equation}
		holds. This condition guarantees inductively that 
	$\Omega_s\le \Omega_0-\tfrac{s/4}{3r-2}n$ and because of $\Omega_s\ge 0$ 
	this means that at some moment 
	we will run out of permissible steps. When this happens we stop
	the procedure and we let 
		\begin{equation}\label{eq:Bdef}
		V(G)=B_1\dcup B_2\dcup \ldots \dcup B_{r}
	\end{equation}
		be the terminal partition. The remainder of this proof is dedicated
	to proving that this partition is $(r, \eps)$-exact.
	If the above procedure lasted for $t$ steps, then
		\[
		\tfrac{t/4}{3r-2}n\le \Omega_0
		\overset{\text{\ref{it:a1}}}{\le}
		(r+10)\eta n^2
	\]
		informs us that 
		\begin{equation} \label{eq:t}
		t\le 4(3r-2)(r+10)\eta n\le 48r^2\eta n\,.
	\end{equation}
		
	In particular, $\eta\ll \eps\ll 1$ allows us to conclude that 
	$t\le\tfrac12 \eps n$. Since only $t$ vertices were moved during the process,
	it follows from this bound and from~\ref{it:a3} as well as~\ref{it:a4}
	that the clauses~\ref{it:r1} and~\ref{it:r2} of Definition~\ref{dfn:reps}
	are satisfied.  
	
	For fixed $i\in [r]$ the current value of $e(A_i)$ can change by at most $n$
	in every step and thus we have 
		\[
		e(B_i)\le e(A^0_i)+tn\le 49r^2\eta n^2\le \eps n^2
	\]
		by~\ref{it:a1} and~\eqref{eq:t}, which shows the validity of~\ref{it:r3}.
	The proof of~\ref{it:r5} is very similar but uses~\ref{it:a6} instead
	of~\ref{it:a1}. We leave the details to the reader. 
	Proceeding similarly with~\ref{it:a5} one can obtain 
		\begin{equation}\label{eq:b-dense}
		e(B_1, B_2)\ge \tfrac 12|B_1||B_2|-\tfrac 12 \eps n^2\,.
	\end{equation}
		
	Let us continue with~\ref{it:r4}. For any two sets $X_1\subseteq B_1$
	and $X_2\subseteq B_2$ we have 
		\begin{align*}
		e(X_1, X_2) & \le e(X_1\cap A^0_1, X_2, \cap A^0_2)+(|B_1\sm A^0_1|+|B_2\sm A^0_2|)n \\
		& \overset{\text{\ref{it:a2}}}{\le} 
		\tfrac 12 |X_1\cap A^0_1||X_2, \cap A^0_2|+\eta n^2+tn 
	\end{align*}
		and in view of~\eqref{eq:t} it follows that
		\begin{equation}\label{eq:x12}
		e(X_1, X_2)\le \tfrac 12 |X_1||X_2|+\tfrac14 \eps n^2\,.
	\end{equation}
		We still need an estimate in the other direction and for this purpose 
	we invoke~\eqref{eq:b-dense} and make two applications of~\eqref{eq:x12},
	thus getting
		\begin{align*}
		e(X_1, X_2) &=e(B_1, B_2)-e(B_1, B_2\sm X_2)-e(B_1\sm X_1, X_2) \\
		&\ge \bigl(\tfrac 12|B_1||B_2|-\tfrac 12 \eps n^2\bigr)
		-\bigl(\tfrac 12|B_1||B_2\sm X_2|+\tfrac 14 \eps n^2\bigr)
		-\bigl(\tfrac 12|B_1\sm X_1||X_2|+\tfrac 14 \eps n^2\bigr) \\
		& =\tfrac 12|X_1||X_2|- \eps n^2\,.
	\end{align*}
		Altogether the pair $(B_1, B_2)$ behaves indeed as demanded by~\ref{it:r4}.
	
	It remains to deal with the local minimum conditions~\ref{it:r6},~\ref{it:r7},~\ref{it:r8},
	and~\ref{it:r9}. The proofs of all four of them are very similar and rely on the 
	property~\eqref{eq:Omega} of the procedure that was used to generate the 
	partition~\eqref{eq:Bdef}. We will only display the proof~\ref{it:r7} here and leave the 
	three other clauses to the reader.
	
	Assume, for instance, that there is a vertex $v\in B_1$ with $d_3(v)<\tfrac{5/3}{3r-2}n$.
	Due to the minimum degree condition imposed on $G$ we must have
		\[
		d_1(v)\ge \tfrac{3r-5}{3r-2}n -|B_2| - \tfrac{5/3}{3r-2}n -\sum_{i=4}^r |B_i|\,.
	\]
		Because of~\ref{it:r1} and~\ref{it:r2} this implies 
		\[
		d_1(v)\ge \tfrac{1/3}{3r-2}n -(r-2)\eps n\,,
	\]
		wherefore 
		\[
		6d_1(v)-d_3(v)>\tfrac{1/4}{3r-2}n\,.
	\]
		Consequently we can perform a $(t+1)^{\text{st}}$ step of our procedure 
	and move $v$ from $B_1$ to $B_3$. This contradicts the supposed maximality 
	of $t$, and thereby~\ref{it:r7} is proved.  
\end{proof}

\section{Refined edge counting}
\label{sec:refine}

Let us start this section with an elementary lemma, the following.

\begin{lemma}\label{lem:C7}
	Every graph $G$ not containing a cycle of length $3$, $5$, or $7$ satisfies 
		\[
		e(G)\le \alpha(G)^2\,.
	\]
	\end{lemma}

\begin{proof}
	We construct recursively a sequence $z_1, \ldots, z_k$ of distinct vertices of $G$
	according to the following rules.
	\begin{enumerate}
	\item[$\bullet$] Let $z_1$ be any vertex of $G$ whose degree is maximal.
	\item[$\bullet$] If at some moment the vertices $z_1, \ldots, z_i$ have already 
		been selected, we ask ourselves whether the set $Q_i$ of all vertices 
		having a distance of at least four from all of them is empty or not.
	\item[$\bullet$] If $Q_i=\varnothing$, we set $k=i$ and terminate the procedure.
	\item[$\bullet$] Otherwise we take a vertex $z_{i+1}\in Q_i$ whose degree is as 
		large as possible.
	\end{enumerate}
	Set $Q_0=V(G)$ and $W_i=Q_{i-1}\sm Q_i$ for $i=1, \ldots, k$. Notice that 
		\[
		V(G)=W_1\dcup \ldots\dcup W_k
	\]
		is indeed a partition, because $Q_0\supseteq Q_1\supseteq \dots \supseteq Q_k=\varnothing$.
	Owing to the maximum degree conditions imposed on the vertices $z_i$ we have 
		\begin{equation} \label{eq:egw}
		2e(G)=\sum_{x\in V(G)}d(x)\le \sum_{i=1}^{k} |W_i|\cdot d(z_i)\,.
	\end{equation}
		We contend that for $i\in [k]$ every vertex $x\in W_i$ has at most 
	distance three from $z_i$. To see this we remark that due to $x\not\in Q_i$
	there has to be an index $j\in [i]$ such that $x$ has distance at most three from $z_j$.
	Moreover, $j<i$ would yield $x\not\in Q_{i-1}$, contrary to $x\in W_i$. Thus we must 
	have $j=i$, as desired.
	
	It follows that we can partition $W_i$ into a set of vertices having distance $0$ or $2$
	from $z_i$ and a set of vertices having distance $1$ or $3$ from $z_i$. Both 
	partition classes are independent sets, for otherwise $G$ would contain an odd 
	cycle of length $3$, $5$, or $7$. 
	
	In particular, we have $|W_i|\le 2\alpha(G)$ for each $i\in [k]$ and in view 
	of~\eqref{eq:egw} we obtain
		\[
		e(G)\le \alpha(G)\sum_{i=1}^k d(z_i)\,.
	\]
		
	Due to their construction any two of the vertices $z_1, \ldots, z_k$ have a distance 
	of at least four. Therefore, their neighbourhoods are mutually disjoint and taken 
	together they form an independent set. Thus we have indeed $e(G)\le \alpha(G)^2$.
\end{proof}

After this little distraction we resume our task of proving Theorem~\ref{thm:main}.
In the light of the work in the two previous sections, it seems desirable to deal
with the case that $G$ admits an exact partition, which will occupy the 
remainder of the present section.

\begin{prop}~\label{prop:exact} 
	Given an integer $r\ge 2$, there exists a real $\eps>0$ such that for 
	every $\delta\le \eps$ every $n$-vertex graph $G$ with $K_{2r}\not\subseteq G$ and 
	$\alpha(G)\le \delta n$ admitting an $(r, \eps)$-exact partition of its vertex set 
	has at most $\bigl(\frac{3r-5}{3r-2}+\delta-\delta^2\bigr)\frac{n^2}{2}$ edges.
\end{prop}

\begin{proof}
	Throughout the arguments that follow we will assume that $\eps$ has been chosen 
	so small that all estimates encountered below hold. Now let $\delta\le \eps$, 
	let $G=(V, E)$ be a $K_{2r}$-free graph on $n$ vertices with $\alpha(G)<\delta n$ 
	and let 
	\[
		V = B_1\dcup \ldots \dcup B_r
	\]
	be an $(r, \eps)$-exact partition of $G$. By lowercase greek letters enclosed 
	in parentheses such as~\ref{it:r1}, \dots,~\ref{it:r9} we shall always mean the 
	corresponding clauses of Definition~\ref{dfn:reps}.  
	
	The statement that follows will often be useful in conjunction with the hypothesis 
	that~$G$ be $K_{2r}$-free.
	
	\begin{clm}\label{clm:I}
		Suppose that $I\subseteq [r]$ and that
		for every $i\in I$ we have a set $X_i\subseteq B_i$ 
		with ${|X_i|\ge \frac{1/15}{3r-2}n}$.
		Then the set $X=\bigcup_{i\in I}X_i$ contains a clique of order $2|I|-1$.
		
		Moreover, if $I$ does not contain both of $1$ and $2$, than $X$ does even 
		contain a clique of order $2|I|$.
	\end{clm}
	
	\begin{proof}
		Let us begin with the ``moreover''-part.
		Intending to apply Lemma~\ref{lem:ab} with $\theta=\frac 12$ and $a=b=|I|$  
		we need to check that for distinct $i, j\in I$ the pair $(X_i, X_j)$
		is {$(16^{-r}, 1)$-dense} and that $\alpha(G)<|X_i|/4^r$.
		The latter is an immediate consequence of $\delta\le \eps\ll 1$. 
		Moreover, if $Y_i\subseteq X_i$ and $Y_j\subseteq X_j$, then 
				\[
			e(Y_i, Y_j)  
			\overset{\text{\ref{it:r5}}}{\ge}
			|Y_i||Y_j|-\eps n^2
			\ge |Y_i||Y_j| - 16^{-r} |X_i| |X_j|\,,
		\]
				as desired. If $1, 2\in I$ we can still apply Lemma~\ref{lem:ab} with 
		$\theta=\frac 12$, but this time with $a=|I|$ and $b=|I|-1$. 
		This is because~\ref{it:r4} allows us to show, in the same way as above, 
		that the pair $(X_1, X_2)$ is $(1/16^r, 1/2)$-dense.  
	\end{proof}
	
	Next we explain how condition~\ref{it:r3} is utilised. 
	
	\begin{clm} \label{clm:ex}
		If $i\in [r]$ and $X\subseteq B_i$, then $e(X)\le \frac{n/60}{3r-2}|X|$. 
	\end{clm}
	
	\begin{proof}
		If $|X|\le \frac{n/60}{3r-2}$ this follows from the trivial bound $e(X)\le |X|^2$.
		On the other hand, if $|X|\ge \frac{n/60}{3r-2}$, then we have 
		\[
			e(X)
			\overset{\text{\ref{it:r3}}}{\le}
			\eps n^2\le \bigl(\tfrac{n/60}{3r-2}\bigr)^2 \le \tfrac{n/60}{3r-2}|X|
		\]
		due to $\eps\ll 1$.
	\end{proof}
	
	\begin{clm} \label{clm:eBi}
		For each $i\in [3, r]$ we have 
				\[
			e(B_i, V)\le (n-|B_i|)|B_i| +\delta n|B_i|\,.
		\]
			\end{clm}
	
	\begin{proof}
		Look at the partition $B_i=B_i^+\dcup B_i^-$ defined by 
				\[
			B_i^+=\Bigl\{x\in B_i\colon |N(x)\sm B_i|\ge n-|B_i|-\tfrac{n/15}{3r-2}\Bigr\}
		\]
				and $B_i^-=B_i\sm B_i^+$. Clearly, we have 
				\begin{equation}\label{eq:ebi1}
			e(B_i, V\sm B_i)\le (n-|B_i|)|B_i|-\tfrac{n/15}{3r-2}|B_i^-|
		\end{equation}
		and Claim~\ref{clm:ex} yields 
				\begin{equation}\label{eq:ebi2}
			e(B_i^-)\le \tfrac{n/60}{3r-2}|B_i^-|\,.
		\end{equation}
				
		Now assume for the sake of contradiction that $B_i$ contains a triangle 
		$uvw$ two of whose vertices, say $v$ and $w$, belong to $B_i^+$.
		Let $X$ denote the common neighbourhood of $u$, $v$, and $w$.
		The definition of $B_i^+$ leads to 
				\[
			|X\cap B_j|\ge |N(u)\cap B_j|-\tfrac{2/15}{3r-2}n
			\overset{\text{\ref{it:r9}}}{\ge}
			\tfrac{13/15}{3r-2}n
		\]
				for $j\in [3, r]\sm\{i\}$ and, similarly, we have $|X\cap B_j|\ge \tfrac{1/15}{3r-2}n$
		for $j=1, 2$ due to~\ref{it:r8}. Thus the assumptions of Claim~\ref{clm:I}
		are satisfied by $I=[r]\sm\{i\}$ and $X$, meaning that $X$ contains a~$K_{2r-3}$.
		But together with the triangle $uvw$ this clique gives us a $K_{2r}$ in $G$, 
		which is absurd. 
		
		This proves that there are no such triangles in $B_i$ and due to $\alpha(G)<\delta n$
		it follows that no vertex in~$B_i$ can have more than $\delta n$ neighbours 
		in $B_i^+$. Therefore we have $e(B_i^+, B_i^-)\le \delta n|B_i^-|$
		and $2e(B_i^+)\le \delta n|B_i^+|$. Taking~\eqref{eq:ebi1} and~\eqref{eq:ebi2}
		into account we can now deduce
				\begin{align*}
			e(B_i, V) &= e(B_i, V\sm B_i)+2e(B_i^+)+2e(B_i^+, B_i^-)+2e(B_i^-) \\
			&\le (n-|B_i|)|B_i| +\delta n |B_i| 
			+\bigl(\tfrac{n/30}{3r-2}+\delta n-\tfrac{n/15}{3r-2}\bigr)|B_i^-|\,,
		\end{align*}
				and in view of $\delta\ll 1$ the desired estimate follows.
	\end{proof}
	
	Before we proceed deriving similar upper bounds for $e(B_1, V)$ and $e(B_2, V)$,
	we record some useful properties of the common neighbourhoods of edges in $B_1$.
	
	\begin{clm} \label{clm:cn}
		Any two vertices $u, v\in B_1$ forming an edge have at least $\frac{4/15}{3r-2}n$
		common neighbours in each of $B_3, \ldots, B_r$, but less than $\frac{1/15}{3r-2}n$ 
		common neighbours in $B_2$.
	\end{clm}
	
	\begin{proof}
		For each $i\in [3, r]$ we have 
				\[
			|N(u)\cap N(v)\cap B_i|\ge |N(u)\cap B_i|+|N(v)\cap B_i|-|B_i|\,,
		\]
				which due to~\ref{it:r2} and~\ref{it:r7} yields
				\[
			|N(u)\cap N(v)\cap B_i|\ge
			\tfrac{10/3}{3r-2}n-\bigl(\tfrac{3}{3r-2}+\eps\bigr)n
			\ge \tfrac{4/15}{3r-2}n\,,
		\]
				as desired.
		If $u$ and $v$ had at least $\frac{1/15}{3r-2}n$ common neighbours in $B_2$,
		we could use Claim~\ref{clm:I} with $I=[r]\sm \{1\}$ to find a 
		$K_{2r-2}$ among the common neighbours of those two vertices, 
		contrary to $K_{2r}\not\subseteq G$.
	\end{proof}
	
	\begin{clm} \label{clm:eB1}
		For $i\in \{1, 2\}$ we have 
				\[
			e(B_i, V)\le |B_i|(n-|B_1|-|B_2|)+\tfrac12 |B_1||B_2|
				+\tfrac 12\delta n (|B_1|+|B_2|)-\tfrac 12\delta^2n^2\,.
		\]
			\end{clm}
	
	\begin{proof}
		Due to symmetry it suffices to prove this for $i=1$ only.
		The vertices in 
				\begin{equation} \label{eq:P}
			P=\Bigl\{ x\in B_1\colon 
				|N(x)\sm B_1|\le n-|B_1|-\tfrac12 |B_2|-\tfrac{1/15}{3r-2}n\Bigr\}
		\end{equation}
				receive special treatment.
		
		\begin{fact} \label{fact:P}
			There is no triangle in $B_1$ two of whose vertices
			are outside~$P$. 
		\end{fact}
		
		\begin{proof}
			Arguing indirectly we assume that $uvw$ is such a triangle.
			By Claim~\ref{clm:cn} no two of the three vertices $u$, $v$, and $w$ can have 
			$\frac{1/15}{3r-2}n$ common neighbours in $B_2$, whence
						\[
				d_2(u)+d_2(v)+d_2(w) < |B_2|+\tfrac{1/5}{3r-2}n\,.
			\]
						On the other hand, by the definition of $P$ we have 
			$d_2(x) > \tfrac12 |B_2|-\frac{1/15}{3r-2}n$ for every $x\in B_1\sm P$
			and together with~\ref{it:r6} this yields 
						\[
				d_2(u)+d_2(v)+d_2(w) > 
				2\bigl(\tfrac12 |B_2|-\tfrac{1/15}{3r-2}n\bigr)+\tfrac{1/3}{3r-2}n
				= |B_2|+\tfrac{1/5}{3r-2}n\,.
			\]
						This contradiction proves Fact~\ref{fact:P}.
		\end{proof}
		
		Since $\alpha(G)<\delta n$, it follows that no vertex in $P$ can have $\delta n$
		neighbours in $B_1\sm P$, which in turn reveals $e(P, B_1\sm P)\le \delta n |P|$.
		Together with the estimate $e(P)\le \frac{n/60}{3r-2}|P|$, which follows from
		Claim~\ref{clm:ex}, this gives 
				\[
			2e(P, B_1\sm P)+2e(P) \le \bigl(2\delta+\tfrac{1/30}{3r-2}\bigr)n|P|
			\le \tfrac{1/15}{3r-2}n|P|
		\]
				and by adding the upper bound on $e(P, V\sm B_1)$ that trivially follows 
		from~\eqref{eq:P} we arrive at 
				\begin{equation} \label{eq:eP}
			e(P, V)+e(P, B_1\sm P)\le |P|(n-|B_1|-\tfrac 12|B_2|)\,.
		\end{equation}
				
		\begin{fact} \label{fact:C7}
			There is no $C_3$, $C_5$, or $C_7$ in $G$ all of whose vertices are in $B_1\sm P$.
		\end{fact}
		
		\begin{proof}
			Assume contrariwise that for some $\ell\in \{3, 5, 7\}$ the vertices 
			in $C=\{v_1, \ldots, v_\ell\}$ form such a cycle. 
			If a vertex $x\in B_2$ is adjacent to $q$ vertices in $C$, 
			then the neighbourhood of~$x$ contains at least $q-\tfrac 12(\ell-1)$ 
			edges of this cycle, whence 
						\[
				e(C, B_2)=\sum_{x\in B_2}d_C(x)\le \tfrac 12(\ell-1)|B_2|+t\,,
			\]
						where $t$ denotes the number of triangles formed by a vertex in $B_2$ 
			and an edge of the cycle. Further, by the second part of Claim~\ref{clm:cn},
			each edge of the cycle can sit in at most $\tfrac{1/15}{3r-2}n$
			such triangles, wherefore $t\le \tfrac{7/15}{3r-2}n$.
			
			On the other hand, each $v\in C$ has at least 
			$\tfrac12 |B_2|-\frac{1/15}{3r-2}n$ neighbours in $B_2$ due to $C\subseteq B_1\sm P$
			and~\eqref{eq:P}, whence
						\[
				e(C, B_2)=\sum_{k=1}^\ell d_2(v_k)\ge \tfrac12 \ell |B_2|-\tfrac{7/15}{3r-2}n\,.
			\]
						By combining all these estimates we infer
						\[
				|B_2|\le \tfrac{28/15}{3r-2}n\,,
			\]
						which, however, violates~\ref{it:r1}. 
			This concludes the proof of Fact~\ref{fact:C7}.
		\end{proof}
		
		Now consider the partition 
				\[
			B_1\sm P = Q\dcup R\dcup S
		\]
				defined by 
				\begin{align*}
			Q &= \Bigl\{ x\in B_1\sm P\colon  d_2(x)\le \tfrac12 (|B_2|+\delta n) \Bigr\}\,, \\
			R &= \Bigl\{ x\in B_1\sm P\colon  \tfrac12 (|B_2|+\delta n)<d_2(x)\le 
				\tfrac{7/4}{3r-2}n \Bigr\}\,, \\
			\text{ and } \quad
			S &= \Bigl\{ x\in B_1\sm P\colon  \tfrac{7/4}{3r-2}n< d_2(x) \Bigr\}\,.
		\end{align*}
				\begin{fact} \label{fact:S}
			There is no edge connecting a vertex in $S$ with a vertex in $B_1$.
		\end{fact}
		
		\begin{proof}
			By~\ref{it:r6} and the definition of $S$ the common neighbourhood of 
			such an edge would intersect $B_2$ in at least
						\[
				 \tfrac{7/4}{3r-2}n+\tfrac{1/3}{3r-2}n-|B_2|
				 \overset{\text{\ref{it:r1}}}{\ge}
				 \tfrac{1/15}{3r-2}n
			\]
						vertices, contrary to Claim~\ref{clm:cn}.
		\end{proof}
		
		\begin{fact} \label{fact:R}
			The set $R\cup S$ is independent.
		\end{fact}
		
		\begin{proof}
			Assume that we have an edge $uv$ both of whose endvertices are in $R\cup S$.
			According to the definitions of $R$ and $S$, the common neighbourhood $J$
			of $u$ and $v$ has at least $\delta n$ vertices in $B_2$ and by 
			$\alpha(G)<\delta n$ there exists an edge $xy$ in $B_2\cap J$.
			
			We will now try to construct a $K_{2r-4}$ in the common neighbourhood 
			$J_*\subseteq J$ of $u$, $v$, $x$, and $y$, which would give a contradiction
			to $K_{2r}\not\subseteq G$. To this end we utilise Claim~\ref{clm:I}
			with $I=[r]\sm \{1, 2\}$ and it remains to show that we have 
			$|B_j\cap J_*|\ge \tfrac{1/15}{3r-2}n$ for every $j\in [3, r]$.
			
			Thanks to Claim~\ref{clm:cn} we already know that $x$ and $y$ have 
			at least $\tfrac{4/15}{3r-2}n$ common neighbours in each $B_j$
			with $j\in [3, r]$, so it suffices to prove 
			$|B_j\cap J|\ge |B_j|-\tfrac{1/5}{3r-2}n$
			instead. For this purpose it is enough to establish 
						\[\tag{$\star$}
				|J\sm (B_1\cup B_2)|\ge n-(|B_1|+|B_2|)-\tfrac{1/5}{3r-2}n\,. 
			\]
						
			Now due to $u, v\in B_1\sm P$ and~\eqref{eq:P} we have 
						\[
				|J\sm B_1|\ge 2\bigl(n-|B_1|-\tfrac 12 |B_2|-\tfrac{1/15}{3r-2}n\bigr)
					-(n-|B_1|)=n-|B_1|-|B_2|-\tfrac{2/15}{3r-2}n
			\]
						and Claim~\ref{clm:cn} tells us that 
						\[
				|J\cap B_2|\le \tfrac{1/15}{3r-2}n\,.
			\]
						It is easily seen that the last two estimates imply~$(\star)$.
		\end{proof}
		
		We will now work towards an upper bound on $e(B_1\sm P, B_2)$. Due to the 
		definitions of~$Q$,~$R$, and~$S$ we have
				\begin{align*}
			e(B_1\sm P, B_2)
			&\le |Q|\cdot\tfrac12(|B_2|+\delta n)+|R|\cdot \tfrac{7/4}{3r-2}n +|S||B_2|\\
			&\overset{\text{\ref{it:r1}}}{\le}
			(|Q|+|R|)\cdot\tfrac12(|B_2|+\delta n)+|R|\cdot \tfrac{4/5}{3r-2}n+|S||B_2|\,.
		\end{align*}
				According to Fact~\ref{fact:R} and $\alpha(G)<\delta n$ we have $|R|\le \delta n-|S|$
		and thus we arrive at
				\begin{align*}
			e(B_1\sm P, B_2)
			&\le \tfrac 12|B_1\sm P|(|B_2|+\delta n)+(\delta n-|S|)\tfrac{4/5}{3r-2}n
				+\tfrac 12 |S|(|B_2|-\delta n) \\
			&= \tfrac 12|B_1\sm P||B_2|+\tfrac 12 \delta n(|B_1\sm P|+|B_2|)
				-\tfrac 12\delta^2n^2 \\
			&\qquad\qquad    +(\delta n-|S|)
				\bigl(\tfrac{4/5}{3r-2}n+\tfrac 12\delta n-\tfrac 12|B_2|\bigr)\,.
		\end{align*}
				Employing~\ref{it:r1} we may weaken this to 
				\begin{equation} \label{eq:BB}
			e(B_1\sm P, B_2)\le \tfrac 12|B_1\sm P||B_2|+\tfrac 12 \delta n(|B_1|+|B_2|)
				-\tfrac 12\delta^2n^2 -2\delta n(\delta n-|S|)\,.
		\end{equation}
				
		Next we learn from Lemma~\ref{lem:C7} and Fact~\ref{fact:C7} that 
		$e(Q\cup R)\le \alpha(Q\cup R)^2$, where $\alpha(Q\cup R)$, the size 
		of the largest independent set in $Q\cup R$, is at most $\delta n-|S|$
		due to Fact~\ref{fact:S} and $\alpha(G)<\delta n$. So in other words we have
		$e(Q\cup R)\le (\delta n-|S|)^2\le \delta n (\delta n-|S|)$. A further application 
		of Fact~\ref{fact:S} leads to the seemingly stronger inequality 
		$e(B_1\sm P)\le \delta n (\delta n-|S|)$ and together with~\eqref{eq:BB} 
		this yields 
				\[
			e(B_1\sm P, B_1\cup B_2\sm P)\le \tfrac 12|B_1\sm P||B_2|
				+\tfrac 12 \delta n(|B_1|+|B_2|)-\tfrac 12\delta^2n^2\,.
		\]
				Adding the trivial upper bound for $e\bigl(B_1\sm P, V\sm (B_1\cup B_2)\bigr)$
		we obtain
				\[
			e(B_1\sm P, V\sm P)\le |B_1\sm P|\bigl(n-|B_1|-\tfrac 12|B_2|\bigr)
				+\tfrac 12 \delta n(|B_1|+|B_2|)-\tfrac 12\delta^2n^2\,.
		\]
				Combined with~\eqref{eq:eP} this shows the desired estimate
				\[
			e(B_1, V)\le |B_1|\bigl(n-|B_1|-\tfrac 12|B_2|\bigr)
				+\tfrac 12 \delta n(|B_1|+|B_2|)-\tfrac 12\delta^2n^2
		\]
				and the proof of Claim~\ref{clm:eB1} is thereby complete.
	\end{proof}
	
	Finally, the addition of the $r$ inequalities provided by the Claims~\ref{clm:eBi} 
	and~\ref{clm:eB1} reveals
		\[
		2e(G)=\sum_{i=1}^r e(B_i, V)\le 2\sum_{1\le i<j\le r}|B_i||B_j|-|B_1||B_2|
			+\delta n \sum_{i=1}^r |B_i| -\delta^2 n^2
	\]
		and Lemma~\ref{lem:ai} leads to
		\[
		2e(G)\le \bigl(\tfrac{3r-5}{3r-2}+\delta-\delta^2\bigr)n^2\,.
	\]
		Thereby Proposition~\ref{prop:exact} is proved. 		
\end{proof}


Now the following should be clear.

\begin{prop}\label{thm:mindeg}
	For every integer $r\ge 2$ there exist an integer $n_0$ and a positive real number $\delta_0$
	such that for every $\delta\le \delta_0$ every graph $G$ on $n\ge n_0$ vertices 
	with $K_{2r}\not\subseteq G$, $\delta(G)\ge\tfrac{3r-5}{3r-2}n$, and $\alpha(G)<\delta n$
	has at most $\bigl(\tfrac{3r-5}{3r-2}+\delta-\delta^2\bigr)\tfrac{n^2}2$ edges.
\end{prop}

\begin{proof}
	Let $\eps>0$ be the number provided by Proposition~\ref{prop:exact}. 
	By plugging it into Proposition~\ref{prop:eps} we obtain some constants
	$n_0\in\NN$ and $\delta_0>0$. Without loss of generality we may suppose that 
	$\delta_0\le \eps$. To check that these two numbers have the desired 
	property we consider any graph $G$ on $n\ge n_0$ vertices satisfying the above 
	conditions for some~$\delta\le \delta_0\le \eps$. 
	
	Now Proposition~\ref{prop:eps} informs us that 
	$G$ has an $(r, \eps)$-exact partition and Proposition~\ref{prop:exact}
	yields the desired upper bound on $e(G)$. 
\end{proof}

The only things which are currently missing from a proof of Theorem~\ref{thm:main}
are that we still need to abolish the minimum degree condition and $n_0$. 
Essentially this can be done in the same way as in Section~\ref{sec:odd}, but for the sake of 
completeness we would like to include a sketch of the argument. 

\begin{proof}[Proof of Theorem~\ref{thm:main}]
	Let $n_0\in\NN$ and $\delta_0\in (0, 1)$ be as obtained by Proposition~\ref{thm:mindeg} 
	and set 
		\begin{equation*}\label{eq:delstar}
		\delta_*=\tfrac{1}{4}\min\bigl(\delta_0^2, n_0^{-2}\bigr)\,.
	\end{equation*}
		
	Due to the blow-up trick it suffices to show the apparently weaker statement that 
	if $\delta\le \delta_*$ and a $K_{2r}$-free graph $G$ on $n$ vertices satisfies 
	$\alpha(G)<\delta n$, then 
	\begin{equation}\label{eq:GG}
		e(G) \le \frac{3r-5}{3r-2}\cdot \frac{n^2+n}2+\frac{(\delta-\delta^2)n^2}2\,.
	\end{equation}

	Assuming again that this estimate fails we take a minimal set $X\subseteq V(G)$
	with 
	\begin{equation} \label{eq:evenXX}
		e(X) > \frac{3r-5}{3r-2}\cdot \frac{|X|^2+|X|}2+\frac{(\delta-\delta^2)n^2}2
	\end{equation}
	and denote the restriction of $G$ to $X$ by $G'$. Observe that $X\ne\varnothing$ and 
	put $n'=|X|$ as well as $\delta'=\delta n/n'$. Again the plan is to apply 
	Proposition~\ref{thm:mindeg} to $G'$ and $\delta'$ and the required estimates 
	$\delta(G')\ge \frac{3r-5}{3r-2}|X|$ as well as $\alpha(G')\le \delta' |X|$ hold
	for same reasons as above. Moreover, in view of 
	\[
		(n')^2\ge 2e(X)>\tfrac{3r-5}{3r-2}\bigl((n')^2+n'\bigr)+(\delta-\delta^2)n^2
		> \tfrac{3r-5}{3r-2}(n')^2 + \tfrac12\delta n^2
	\]
	we have, e.g., 
	\begin{equation}\label{eq:n'big}
		n'>\sqrt{\delta} n/2\,.
	\end{equation} 

	Thus $\delta'<2\sqrt{\delta}\le 2\sqrt{\delta_*}\le \delta_0$, meaning that 
	$\delta'$ is indeed sufficiently small. Moreover, since $\delta n>\alpha(G)\ge 1$,
	the estimate \eqref{eq:n'big} does also imply 
	\[
		n'>\frac 1{2\sqrt{\delta}}\ge \frac 1{2\sqrt{\delta_*}}\ge n_0\,,
	\]
	or in other words that $G'$ is still sufficiently large. 
	
	So altogether Proposition~\ref{thm:mindeg} implies 
	\[
		e(X)\le \frac{3r-5}{3r-2}\cdot \frac{|X|^2}2+\frac{\delta' n'\cdot n'-(\delta'n')^2}2
			< \frac{3r-5}{3r-2}\cdot \frac{|X|^2+|X|}2+\frac{\delta n\cdot n-(\delta n)^2}2\,,
	\]
	contrary to~\eqref{eq:evenXX}. This concludes the proof of~\eqref{eq:GG} and, hence, the 
	proof of Theorem~\ref{thm:main}.
\end{proof}

\end{document}